\documentclass[10pt,a4paper]{article}
\usepackage[utf8]{inputenc}
\usepackage[english]{babel}
\usepackage[T1]{fontenc}
\usepackage{amsmath}
\usepackage{amsthm}
\usepackage{dsfont}
\usepackage{amsfonts}
\usepackage{amssymb}
\usepackage{graphicx}
\usepackage{tabularx}
\usepackage{lmodern}
\usepackage[colorlinks=true,linkcolor=blue]{hyperref}
\usepackage{relsize}
\usepackage{geometry}
\usepackage{mathtools}
\usepackage{enumitem}
\usepackage{tikz}
\usetikzlibrary{decorations.text}

\newcommand*{\PP}{\mathbb{P}} 
\newcommand*{\E}{\mathbb{E}} 
\newcommand*{\V}{\mathrm{Var}} 
\newcommand*{\R}{\mathbb{R}} 
\newcommand*{\N}{\mathbb{N}} 
\newcommand*{\Z}{\mathbb{Z}} 
\newcommand*{\C}{\mathbb{C}} 

\newtheoremstyle{theorema}%
    {4pt}
    {4pt}
    {\slshape}
    {}
    {\bfseries}
    {.}
    {.5em}
    {}
\theoremstyle{theorema} 
\newtheorem{theoreme}{Theorem}

\newtheorem{prop}[theoreme]{Proposition}
\newtheorem{lem}[theoreme]{Lemma}

\newtheorem{corollaire}[theoreme]{Corollary}

\theoremstyle{definition} 
\newtheorem{definition}[theoreme]{Definition}
\newtheorem{exemple}[theoreme]{Example}

\theoremstyle{remark} 
\newtheorem*{remarque}{Remark}

\title{Characteristic polynomials of modified permutation matrices at microscopic scale}

\author{Valentin B\textsc{ahier} \thanks{valentin.bahier@math.univ-toulouse.fr, Institut de Math\'ematiques de Toulouse, 118 route de Narbonne, F-31062 Toulouse Cedex 9, France.}}

\date{\today}

\begin{document}
\maketitle

\begin{abstract}
We study the characteristic polynomial of random permutation matrices following some measures which are invariant by conjugation, including Ewens' measures which are one-parameter deformations of the uniform distribution on the permutation group. We also look at some modifications of permutation matrices where the entries equal to one are replaced by i.i.d uniform variables on the unit circle. Once appropriately normalized and scaled, we show that the characteristic polynomial converges in distribution on every compact subset of $\mathbb{C}$ to an explicit limiting entire function, when the size of the matrices goes to infinity. Our findings can be related to results by Chhaibi, Najnudel and Nikeghbali on the limiting characteristic polynomial of the Circular Unitary Ensemble \cite{chhaibi2017circular}.
\end{abstract}

\section{Introduction}

\subsection{Convergence of characteristic polynomials}
\label{sec:intro1}

Characteristic polynomials of random matrices have drawn much interest the last few decades. These objects encode the information of the whole spectrum of matrices. Moreover, in the case of unitarily invariant matrices (as Gaussian Unitary Ensemble or Circular Unitary Ensemble), the characteristic polynomial is believed to have a similar microscopic behavior as holomorphic functions which appear in number theory, as the Riemann zeta function. The characteristic polynomial of random matrices is also related to Gaussian fields, including the Gaussian multiplicative chaos introduced by Kahane \cite{kahane1985chaos}.

On the macroscopic scale, Keating and Snaith \cite{keating2000random}, Hugues Keating and O-Connell \cite{hughes2001characteristic}, and then Bourgade Hugues Nikeghbali and Yor \cite{bourgade2008characteristic} study the logarithm of characteristic polynomial of unitary matrices following the Haar distribution, and prove in particular that its real and imaginary parts normalized by $\sqrt{\frac{1}{2} \log n}$ converge jointly in law to independent centred and reduced Gaussian random variables. \\
Hambly, Keevash, O-Connell and Stark \cite{hambly2000characteristic} give a similar result for permutation matrices following the uniform measure. Zeindler \cite{zeindler2010permutation} \cite{zeindler2013central} generalizes this result for permutation matrices under Ewens measures, considering more general class functions than the characteristic polynomial, the so-called multiplicative class functions. 
Dehaye and Zeindler \cite{dehaye2013averages}, and Dang and Zeindler \cite{dang2014characteristic} extend the study to some Weyl groups, and some wreath products involving the symmetric group.

On the microscopic scale, Chhaibi, Najnudel and Nikeghbali \cite{chhaibi2017circular} show that the characteristic polynomial of unitary matrices following the Haar measure, suitably renormalized, converges to a limiting entire function. With the coupling of virtual isometries introduced by Bourgade, Najnudel and Nikeghbali \cite{bourgade2012unitary}, the authors get an almost sure convergence.
Chhaibi, Hovhannisyan, Najnudel, Nikeghbali, and Rodgers \cite{chhaibi2017limiting} extend the study to the special orthogonal group, the symplectic group, and give a related result for the Gaussian Unitary Ensemble.

Our motivation in this paper is to prove similar results on the characteristic polynomial of some particular unitary matrices related to random permutations. More precisely:
\begin{itemize}
	\item We focus on matrices belonging to two particular subgroups of the unitary group: the set of permutation matrices, and the wreath product $S^1 \wr \mathfrak{S}_n$ (which can be seen as the set of permutation matrices where entries equals to one are replaced by complex numbers of modulus one).
	\item We tackle a large family of measures on the symmetric group, which are invariant by conjugation and verify a certain property of decay over the cycle lengths. This family includes the family of Ewens measures, as we shall see.
	\item We introduce a coupling method for generating sequences of modified permutations under these particular measures, by analogy of the notion of virtual isometries introduced in \cite{bourgade2012unitary}. This coupling provides an almost sure convergence in our main result given below.
\end{itemize}

\subsubsection*{Notations}

For all events $A$ and all random variables $y$, we will denote by $\PP_y (A) := \E (\mathds{1}_A \ \vert \ y)$ the conditional expectation of $\mathds{1}_A$ given $y$. \\
We will write $X_n \overset{d}{\underset{n\to\infty}{\longrightarrow}} X$ for the convergence in distribution of the sequence of random variables $(X_n)$ to the random variable $X$. \\
We will use the arrow $\Longrightarrow$ to denote the convergence in law on the space of continuous functions from $\C$ to $\C$ equipped with the topology of uniform convergence on compact sets. \\
Finally, for all real numbers $x$, $\{x\}=x-\lfloor x \rfloor$ will denote the fractional part of $x$, and $\Vert x \Vert$ the distance from $x$ to the nearest integer. 

\subsection{Main results and outline of the paper}
\label{sec:intro2}

Let $\sigma = (\sigma_n)_{n\geq 1}$ be a \emph{random virtual permutation} (we give the definition in the next section). Let $(M_n)_{n\geq 1}$ be the sequence of random permutation matrices associated to $\sigma$, that is to say for each $n$ we define $M_n$ as the $n \times n$ matrix whose coordinates are given by 
\begin{equation}
\forall 1\leq i,j \leq n, \ (M_n)_{i,j} := \mathds{1}_{i = \sigma_n (j)}.
\end{equation}
Let $(\widetilde{M}_n)_{n\geq 1}$ be the \emph{random modified virtual permutation} generated by $\sigma$ and a sequence $(u_j)_{j\geq 1}$ of i.i.d uniform variables on the unit circle independent of $\sigma$ (see Corollary~\ref{cor:couplage}).

For all $n\in \N^*$ and $z \in \C$, we consider the characteristic polynomials of $M_n$ and $\widetilde{M}_n$, respectively defined by
\begin{equation}\label{eq:Z}
Z_n (z) := \det (zI - M_n)
\end{equation}
and 
\begin{equation}\label{eq:Ztilde}
\widetilde{Z}_n (z) := \det (zI - \widetilde{M}_n).
\end{equation}

We are interested in the behavior as $n$ goes to infinity of 
\begin{equation}\label{eq:xi1}
\widetilde{\xi}_n (z) = \frac{\widetilde{Z}_n (\mathrm{e}^{2i \pi z/n})}{\widetilde{Z}_n (1)}
\end{equation}
and 
\begin{equation}\label{eq:xialpha1}
\xi_{n,\alpha} (z) = \frac{Z_n \left(\mathrm{e}^{2i \pi \left(\frac{z}{n} + \alpha\right)}\right)}{Z_n \left( \mathrm{e}^{2i \pi \alpha} \right)}
\end{equation}
where $\alpha$ is an irrational number between $0$ and $1$.

It follows from the preliminary results given in the next section that $\widetilde{\xi}_n$ can be written with the help of the normalized cycle-lengths $(y_j^{(n)})_{j\geq 1}$ (see \eqref{eq:ell} and \eqref{eq:yjn}) and the uniform variables $(u_j)_{j\geq 1}$ as
\begin{equation}\label{eq:xi2}
\widetilde{\xi}_n (z) = \prod_{\substack{j\geq 1 \\ \ell_{n,j}>0}} \frac{\mathrm{e}^{2i\pi z y_j^{(n)}} - u_j}{1-u_j}, 
\end{equation}
and similarly, $\xi_{n,\alpha}$ can be written as
\begin{equation}\label{eq:xialpha2}
\xi_{n,\alpha } = \prod_{\substack{j\geq 1 \\ \ell_{n,j}>0}} \frac{\mathrm{e}^{2i\pi \left(\frac{z}{n} + \alpha \right)\ell_{n,j}} -1}{\mathrm{e}^{2i\pi \alpha \ell_{n,j}} -1} .
\end{equation}
We will give more details about these expressions in Section~\ref{sec:quotient}.

Finally, let us recall that the \textbf{type} of any real number $x$ is defined by
\begin{equation}\label{eq:type}
\eta=\sup\{ \gamma \in \R : \ \liminf_{n\to +\infty} n^\gamma \parallel n x \parallel = 0 \} \in \R_+ \cup \{+\infty\}. 
\end{equation}
We say that $x$ is of \textbf{finite type} if $\eta$ is finite. A basic property is that if $x$ is irrational then its type is greater or equal to one (and can be infinite). See \emph{e.g.} \cite{hambly2000characteristic} for a little more details about finite type.

The main result of the present paper is the following:

\begin{theoreme}\label{thm:convergence}
Let $p$ be a distribution on $\nabla^\prime$ (see \eqref{eq:nablap}) with exponential decay (in the sense of Definition~\ref{def:expdec}). Assume that $\sigma$ is generated by a random vector $(y_1, y_2, \dots )$ following $p$. Then we have the following convergences:
\begin{enumerate}[label=(\roman*)]
	\item Almost surely, $\widetilde{\xi}_n$ converges uniformly on every compact set to an entire function $\widetilde{\xi}_\infty$ defined by 
\[\widetilde{\xi}_\infty (z) = \prod_{j=1}^{+\infty} \frac{\mathrm{e}^{2i\pi z y_j} - u_j}{1-u_j} .\]
	\item Assume $\alpha$ is an irrational number of finite type. Then 
\[\xi_{n,\alpha } \underset{n\to \infty}{\Longrightarrow} \widetilde{\xi}_\infty \]
where $\widetilde{\xi}_\infty$ is the same entire function as above. 
\end{enumerate}
\end{theoreme}

\begin{remarque}
Without the coupling of the modified virtual permutation, the theorem still holds replacing the first point by: 
\begin{enumerate}
	\item[(\rm{i}${}^\prime$)]  
	\[\widetilde{\xi}_n  \underset{n\to \infty}{\Longrightarrow} \widetilde{\xi}_\infty . \]
\end{enumerate}
\end{remarque}

\begin{remarque}
Note that the parameter $\alpha$ is not allowed to be rational, otherwise some denominators in the product expression of $\xi_{n,\alpha}$ could be zeros. Moreover, heuristically, the motivation to take $\alpha$ irrational of finite type is to avoid a too fast accumulation of small denominators.
\end{remarque}

The article is organized as follows: In Section~\ref{sec:gen}, we set out our preliminary definitions and results for generating sequences of random permutations and sequences of modified permutation matrices. In Section~\ref{sec:quotient}, we give a proof of Theorem~\ref{thm:convergence} by showing the first point in Subsection~\ref{sec:quotientMPM}, and the second point in Subsection~\ref{sec:quotientPM}. These two subsections are mutually independent. In Section~\ref{sec:properties}, we give some estimates on the limiting function $\widetilde{\xi}_\infty$, and compare our results to the unitary case presented in \cite{chhaibi2017circular}. Finally, in Section~\ref{sec:moregeneral}, we extend the study to more general central measures, removing the restriction to $\nabla^\prime$ for the support of their corresponding distributions on $\nabla$ (see \eqref{eq:nabla}).

\section{Generating random permutations}
\label{sec:gen}

Before giving the construction of the random permutations we will deal with, let us recall the few following definitions and facts: \\
A \textbf{virtual permutation} is a sequence $(\sigma_n)_{n\geq 1}$ where for each $n$, $\sigma_n$ is an element of $\mathfrak{S}_n$ which can be derived from $\sigma_{n+1}$ by simply removing the element $n+1$ from the decomposition into disjoint cycles of $\sigma_{n+1}$. Let $\mathfrak{S}$ denote the space of virtual permutations. 
There is a canonical projection of a measure $\mathcal{L}$ on $\mathfrak{S}$ to a measure $\mathcal{L}_n$ on $\mathfrak{S}_n$. We call $\mathcal{L}$ \textbf{central} if and only if each $\mathcal{L}_n$ is central, that is to say
\[\forall n\geq 1 , \ \forall \tau \in \mathfrak{S}_n , \ \sigma_n \overset{d}{=} \tau \sigma_n \tau^{-1}.\]

For each $n$, it is easy to notice that every central measure on $\mathfrak{S}_n$ can be fully described by a distribution on the set 
\[\nabla^{(n)}:=\left\{(\ell_1, \dots , \ell_n) \in \N^n : \ \ell_1 \geq \ell_2 \geq \cdots \geq \ell_n , \sum\limits_{j=1}^n \ell_j = n  \right\}\]
of partitions of the integer $n$, and conversely, in such a way that there is a one-to-one correspondence. A highly less obvious result (Theorem~2.3 in \cite{olshanski2011random}) is that there exists a natural one-to-one correspondence between the central measures on $\mathfrak{S}$ and the probability measures on 
\begin{equation}\label{eq:nabla}
\nabla:= \left\{ (\lambda_1, \lambda_2 , \dots ) \in [0,1]^\infty : \ \lambda_1 \geq \lambda_2 \geq \dots , \ \sum\limits_{j=1}^{+\infty} \lambda_j \leq 1 \right\}.
\end{equation}

The following definition introduces a new notion which specifies the family of measures we are going to consider in the paper. 
\begin{definition}\label{def:expdec}
Let $p$ be a probability measure on $\nabla$.
\begin{itemize}
	\item We say that $p$ is a \textbf{measure with exponential decay} if it satisfies the following property: There exists $r\in (0,1)$ and $\nabla_1 \subset \nabla$ with $p(\nabla_1)=1$, such that for all $\lambda =(\lambda_1, \lambda_2 , \dots ) \in \nabla_1$, 
\begin{equation}\label{eq:expdecayrate}
\exists C>0, \ \forall j \geq 1, \ \lambda_j\leq C r^j.
\end{equation}
	\item We say that a distribution on $\mathfrak{S}$ is a \textbf{central measure with exponential decay} if its corresponding distribution on $\nabla$ is a measure with exponential decay.
\end{itemize}
\end{definition} 

\begin{exemple}
The \textbf{Ewens measure} \cite{ewens1972sampling} of any arbitrary parameter $\theta>0$ on $\mathfrak{S}$, denoted by $\mathrm{Ewens}(\theta)$, is a central measure with exponential decay. \\
Indeed, first recall that, given $\theta>0$, one can define $\mathrm{Ewens}(\theta)$ on $\mathfrak{S}$ thanks to the family of Ewens measures of parameter $\theta$ on $\mathfrak{S}_n$, $n\geq 1$, denoted by $\mathrm{Ewens}(n,\theta)$, and defined by the probability functions
\begin{equation}
\forall \sigma \in \mathfrak{S}_n, \ \PP_\theta^{(n)} (\sigma ) = \frac{\theta^{K(\sigma)}}{\theta (\theta + 1) \cdots (\theta + n-1 )},
\end{equation}
where $K(\sigma)$ denotes the total number of cycles of $\sigma$ once decomposed into disjoint cycles. More precisely, the sequence of measures $(\mathrm{Ewens}(n,\theta))_{n\geq 1}$ is coherent with the projections $\mathfrak{S}_{n+1} \to \mathfrak{S}_n$. In other words, if $\sigma_{n+1}$ follows $\mathrm{Ewens}(n+1,\theta)$, then the random permutation obtained by removing the element $n+1$ from the cycle-decomposition of $\sigma_{n+1}$ follows $\mathrm{Ewens}(n,\theta)$. \\
For each $\theta>0$, the fact that $\mathrm{Ewens}(\theta)$ is central on $\mathfrak{S}$ immediately derives from the fact that $\mathrm{Ewens}(n,\theta)$ is central on $\mathfrak{S}_n$ for all $n$. It is also well-known that the corresponding distribution on $\nabla$ of the central measure $\mathrm{Ewens}(\theta)$ is the \textbf{Poisson-Dirichlet distribution} of parameter $\theta$ (denoted by $\mathrm{PD}(\theta)$). \\
Let $y=(y_1, y_2, \dots)$ be a random vector following $\mathrm{PD}(\theta)$. We know that $Y$ has the same distribution as the order statistics $(Y_{(1)}, Y_{(2)}, \dots )$ of the random vector $Y=(Y_1, Y_2, \dots )$ defined as follows: let $(V_k)_{k\geq 1}$ be a sequence of i.i.d $\mathrm{Beta}(1,\theta)$ random variables (with density function given by $x\mapsto \theta (1-x)^{\theta -1} \mathds{1}_{(0,1)}(x)$). For all $j\geq 2$, define $Y_j:=V_j \prod\limits_{k=1}^{j-1} (1- V_k)$, and $Y_1:=V_1$. The distribution of $Y$ is called $\mathrm{GEM}(\theta )$. In the literature, this method for generating such a vector $Y$ with i.i.d random variables $(V_k)$ is called \textbf{residual allocation model} \cite{patil1977diversity} or \textbf{stick-breaking process} \cite{kerov1997stick}. With this representation it is easy to compute that for all $j$, 
\[\E (Y_j) =  \frac{1}{1+\theta} \left( \frac{\theta}{1+\theta} \right)^{j-1} \leq r_\theta^j \]
with $r_\theta:= \frac{1 \vee \theta}{1+\theta} <1$. Hence for any arbitrary $r\in (r_\theta,1)$,
\[\PP(Y_j > r^j) \leq \frac{\E (Y_j)}{r^j} \leq \left(\frac{r_\theta}{r}\right)^j,\]
which is summable over $j$, and then the Borel-Cantelli lemma applies and gives that the number of $j$ such that $Y_j > r^j$ is almost surely finite. In other words, there exists a random number $C>0$ such that for all $j$, $Y_j \leq C r^j$. Finally, coming back to $y$ it remains to see that the same kind of inequality holds for its coordinates, which is a direct consequence of the fact that for all $j$ we have $Y_{(j)} \leq \left((C r^k)_{k\geq 1}\right)_{(j)} = C r^j$. Then the Ewens measure is a central measure with exponential decay.
\end{exemple}

\begin{remarque}
Note that the Ewens measures are particular central measures whom corresponding distributions on $\nabla$ are supported on 
\begin{equation}\label{eq:nablap}
\nabla^\prime :=\left\{ (\lambda_1, \lambda_2 , \dots ) \in [0,1]^\infty : \ \lambda_1 \geq \lambda_2 \geq \dots , \ \sum\limits_{j=1}^{+\infty} \lambda_j = 1 \right\} \subset \nabla.
\end{equation}
In the main body of the paper, we focus on central measures with exponential decay on $\mathfrak{S}$ whose corresponding distributions on $\nabla$ are supported on $\nabla^\prime$. 
\end{remarque}

Now, let us present the coupling we consider for generating random permutations, which is highly inspired from \cite{tsilevich1997distribution}, \cite{najnudel2013distribution}, and \cite{najnudel2014flow}. \\
Let $\lambda = (\lambda_j)_{j\geq 1}$ be an element of $\nabla^\prime$, and let $E_\lambda = \bigsqcup\limits_{j=1}^\infty \mathcal{C}_j$ be the disjoint union of circles $\mathcal{C}_j$, where for all $j$, $\mathcal{C}_j$ has perimeter $\lambda_j$. Let $x=(x_k)_{k\geq 1} \in (E_\lambda )^{\infty}$. For all $n\geq 1$, one defines a permutation $\sigma_n (\lambda , x) \in \mathfrak{S}_n$ as follows: for all $k \in \{ 1, \dots , n\}$, there exists a unique $j$ such that $x_k \in \mathcal{C}_j$. Let us follow the circle $\mathcal{C}_j$, counterclockwise, starting from $x_k$. The image of $k$ by $\sigma_n (\lambda ,x)$ is the index of the first point in $\{x_1, \dots , x_n\}$ we encounter after $x_k$. In particular, if $x_k$ is the only point in $\mathcal{C}_j$ and $\{x_1, \dots , x_n\}$, then $k$ is a fixed point of $\sigma_n (\lambda ,x)$, because starting from $x_k$ we do a full turn of the circle $\mathcal{C}_j$ before encountering $x_k$ again. To illustrate, if the $\lambda_j$ equal $2^{-j}$ and if the six first $x_k$ are distributed on $E_\lambda$ as shown
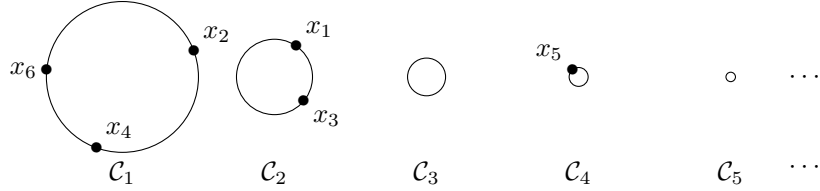
\begin{center}
\begin{tikzpicture}
\draw (0,0) circle (1) ;
\draw (2,0) circle (0.5) ;
\draw (4,0) circle (0.25) ;
\draw (6,0) circle (0.125) ;
\draw (8,0) circle (0.0625) ;
\draw (9,0) node {$\cdots$} ;
\draw (0,-1) node[below] {$\mathcal{C}_1$} ;
\draw (2,-1) node[below] {$\mathcal{C}_2$} ;
\draw (4,-1) node[below] {$\mathcal{C}_3$} ;
\draw (6,-1) node[below] {$\mathcal{C}_4$} ;
\draw (8,-1) node[below] {$\mathcal{C}_5$} ;
\draw (9,-1) node[below] {$\cdots$} ;
\draw (20:1) node {$\bullet$} node[above right] {$x_2$};
\draw (175:1) node {$\bullet$} node[left] {$x_6$};
\draw (-110:1) node {$\bullet$} node[above right] {$x_4$};
\draw (55:0.5)+(2,0) node {$\bullet$} node[above right] {$x_1$};
\draw (-40:0.5)+(2,0) node {$\bullet$} node[below right] {$x_3$};
\draw (130:0.125)+(6,0) node {$\bullet$} node[above left] {$x_5$};
\end{tikzpicture}
\end{center}
then, 
\begin{align*}
\sigma_1 (\lambda , x) & = (1) \\
\sigma_2 (\lambda , x) & = (1)(2) \\
\sigma_3 (\lambda , x) & = (1 \ 3)(2) \\
\sigma_4 (\lambda , x) & = (1 \ 3)(2 \ 4) \\
\sigma_5 (\lambda , x) & = (1 \ 3)(2 \ 4)(5) \\
\sigma_6 (\lambda , x) & = (1 \ 3)(2 \ 6 \ 4)(5).
\end{align*}
The key feature of this construction is highlighted in the following proposition.
\begin{prop}\label{prop:couplage}
The sequence $\sigma_\infty (\lambda, x)=(\sigma_n (\lambda, x))_{n\geq 1}$ is a virtual permutation. Moreover, if $\lambda$ follows any arbitrary distribution $p$ on $\nabla^\prime$, and if conditionally on $\lambda$ the points $x_k$ are i.i.d following the uniform distribution on $E_\lambda$, then $\sigma_\infty (\lambda, x)$ follows the central measure on $\mathfrak{S}$ corresponding to $p$.
\end{prop}

\begin{exemple}
If $\lambda$ follows the $\mathrm{PD}(\theta )$ distribution and if conditionally on $\lambda$ the points $x_k$ are i.i.d random variables uniformly distributed on $E_\lambda $, then $\sigma_\infty (\lambda, x)$ follows $\mathrm{Ewens}(\theta )$.
\end{exemple}
 
Let $p$ be a distribution on $\nabla^\prime$. Let $y=(y_1, y_2, \dots)$ be a random vector following $p$ and let $E_y$ be the disjoint union of circles $\mathcal{C}_j$ of perimeters $y_j$. Assume that conditionally given $y$, the $x_k$ are i.i.d random variables uniformly distributed on $E_y$. Finally, introduce the array of random variables $(\ell_{n,j})_{n,j\geq 1}$ defined by
\begin{equation}\label{eq:ell}
\ell_{n,j} := \# \{ k\in \{1, \dots , n\}: \ x_k \in \mathcal{C}_j  \}, 
\end{equation}
and denote 
\begin{equation}\label{eq:yjn}
y_j^{(n)}:= \frac{\ell_{n,j}}{n}.
\end{equation}
Then, as a consequence of Proposition~\ref{prop:couplage}, almost surely, $(y_1^{(n)}, y_2^{(n)} , \dots )$ converges in distribution to $y$. Moreover, conditionally on $y$, for all $j$,
\begin{equation}
y_j^{(n)} = \frac{1}{n} \sum_{k=1}^n \mathds{1}_{x_k \in \mathcal{C}_j } \overset{a.s.}{\underset{n\to \infty} {\longrightarrow}} y_j 
\end{equation}
by the strong law of large numbers.

In this paper we also consider some modifications of permutation matrices, that we will call \textbf{modified permutation matrices}, which are permutation matrices where the entries equal to one are replaced by complex numbers of modulus one. The set of modified permutation matrices of size $n$ has a group structure and can be identified to the wreath product $S^1 \wr \mathfrak{S}_n$, where $S^1$ denotes the unit circle. Let us denote by $\mathcal{T}_n$ the subset of matrices of $S^1 \wr \mathfrak{S}_n$ which do not have $1$ as an eigenvalue.
The next lemma provides a construction of sequences of elements of $\mathcal{T}_n$, $n\geq 1$, by analogy to the notion of \textbf{virtual isometries} introduced by Bourgade, Najnudel and Nikeghbali in \cite{bourgade2012unitary}.

\begin{lem}\label{lem:proj}
For all $n\geq 1$, for all $M\in \mathcal{T}_{n+1}$, there exists a unique $N\in \mathcal{T}_n$ such that 
\begin{equation}
\mathrm{rank} \left( M -
\left(\begin{array}{@{}c|c@{}}
  \begin{matrix}
   &  & \\
   & N & \\
   & & 
  \end{matrix}
  &   \begin{matrix}
   0 \\
   \vdots \\
   0 
  \end{matrix}
   \\
\hline
  \begin{matrix}
   0 & \cdots & 0
  \end{matrix} & 1
\end{array}\right)
\right) = 1. 
\end{equation}
Moreover, the permutation corresponding to $N$ derives from the one of $M$ by removing the element $n+1$ from its cycle-decomposition.
\end{lem}
Before proving this result, let us give an insight with the basic example $M=\begin{pmatrix}
0 & 0 & z_3 \\
z_1 & 0 & 0 \\
0& z_2 & 0
\end{pmatrix} \in \mathcal{T}_3$. \\ 
Let $N \in \mathcal{M}_2 (\C )$. First observe that one can write
\[M-\mathrm{diag} (N,1) =  \left(\begin{array}{@{}c|c@{}}
  \begin{matrix}
   &  & \\
   & Q & \\
   & & 
  \end{matrix}
  &   \begin{matrix}
   z_3 \\
    \\
   0 
  \end{matrix}
   \\
\hline
  \begin{matrix}
   0 & & z_2
  \end{matrix} & -1
\end{array}\right) \] 
where $Q = \begin{pmatrix}
0 & 0  \\
z_1 & 0 
\end{pmatrix} - N$. Then $\mathrm{rank} (M-\mathrm{diag} (N,1)) = 1$ implies that the first column and the second row of 
$M-\mathrm{diag} (N,1)$ are zeros, notably $Q_{1,1} = Q_{2,1} = Q_{2,2} = 0$. Moreover, from this same rank condition we deduce
$\det \begin{pmatrix}
Q_{1,2} & z_3 \\
z_2 & -1
\end{pmatrix}=0$,
\emph{i.e.} $Q_{1,2}=-z_2 z_3$. \\
Conversely, the matrix $N:= \begin{pmatrix}
0 & z_2 z_3 \\
z_1 & 0
\end{pmatrix}$ satisfies $\mathrm{rank} (M-\mathrm{diag} (N,1)) = 1$, and $N$ lies in $\mathcal{T}_2$ since $z_1 z_2 z_3 \neq 1$ by assumption on $M$.

\begin{proof}[Proof of Lemma~\ref{lem:proj}]
Let $n\geq 1$ and $M\in \mathcal{T}_{n+1}$.
Write $(w_1 \ w_2 \ \dots \ w_\ell \ w_{\ell +1}\!\!=\!\!n+1)$ the cycle of the corresponding permutation of $M$ containing the element $n+1$. There exist $z_1, \dots , z_\ell $ and $z_{n+1}$ some complex numbers of modulus one such that for all $k \in \{1, \dots , \ell \}$, $M e_{w_k} = z_k e_{w_{k+1}}$ and $M e_{n+1} = z_{n+1} e_{a_1}$ where $(e_1, \dots , e_{n+1})$ is the canonical basis of $\C^{n+1}$.

Denote by $M^{[n]}$ the top-left minor of size $n$ of $M$. Let $N \in \mathcal{M}_n (\C )$.
\begin{itemize}
	\item If $\ell =0$ (\emph{i.e.} $n+1$ is a fixed point of the associated permutation), then $z_{n+1}$ is an eigenvalue of $M$. By hypothesis, this implies $z_{n+1} \neq 1$. Hence $\mathrm{rank} (M - \mathrm{diag}(N,1))=1$ if and only if $N=M^{[n]}$ (since $z_{n+1}-1$ is the only non-zero entry of the last row and last column of $M$). Moreover in this case, as $M=\mathrm{diag} (N,z_{n+1})$ we have $N \in \mathcal{T}_n$, and the procedure amounts to remove the fixed point $n+1$ from the associated permutation of $M$.
	\item If $\ell \geq 1$, then $M e_{w_\ell} = z_\ell e_{n+1}$ and $M e_{n+1} = z_{n+1} e_{w_1}$ with $w_1 \neq n+1 \neq w_\ell$. The $(\ell + 1)$-th roots of $z_1 \dots z_\ell z_{n+1}$ are eigenvalues of $M$. By hypothesis, it follows $z_1 \dots z_\ell z_{n+1} \neq 1$. Moreover, $\mathrm{rank} (M - \mathrm{diag}(N,1))=1$
if and only if 
$N=M^{[n]} + z_\ell z_{n+1} E_{w_1 w_\ell}$ where $E_{ij}$ is the $n$-by-$n$ matrix with $1$ in row $i$ column $j$, and zeros elsewhere. In this case, $N e_{w_k} = z_k e_{w_{k+1}}$ for all $k\in \{ 1 , \dots \ell -1\}$ (not considered when $\ell =1$) and $N e_{w_\ell} = z_\ell z_{n+1} e_{w_1}$, so that the $\ell$-th roots of $z_1 \dots z_\ell z_{n+1}$ are eigenvalues of $N$ (the corresponding cycle is $(w_1 \ w_2 \ \dots \ w_\ell )$). As $z_1 \dots z_\ell z_{n+1} \neq 1$ we deduce $N \in \mathcal{T}_n$.
\end{itemize} 
\end{proof}

\begin{definition}
We say that a sequence of matrices $(\widetilde{M}_n)_{n\geq 1}$ is a \textbf{modified virtual permutation} if for all $n$, $\widetilde{M}_n \in \mathcal{T}_n$ and $\mathrm{rank} (\widetilde{M}_{n+1} - \mathrm{diag}(\widetilde{M}_n , 1))=1$.
\end{definition}

\begin{remarque}
Note that every modified virtual permutation is in particular a virtual isometry.
\end{remarque}

\begin{prop}\label{prop:couplage2}
Let $(\widetilde{M}_n)_{n\geq 1}$ be a modified virtual permutation. There exists a virtual permutation $(\sigma_n)_{n\geq 1}$ such that, for all $n\geq 1$, $\widetilde{M}_n$ corresponds to the permutation $\sigma_n$ and has a characteristic polynomial of the form
\[\chi_{\widetilde{M}_n} (X) := \det (XI-\widetilde{M}_n) = \prod_{\substack{j\geq 1 \\ \ell_{n,j}>0}} (X^{\ell_{n,j}} -u_j),\]
where $(u_j)_{j\geq 1}$ is a sequence of elements of $S^1\setminus \{1\}$ and the $\ell_{n,j}$ denote the cycle-lengths of $\sigma_n$. Moreover, for all $j$ and $n$ such that $\ell_{n,j}>0$, $u_j$ can be defined as the product of the non-zero entries of $\widetilde{M}_n$ corresponding to the cycle $j$ of $\sigma_n$.
\end{prop}
\begin{proof}
Let $n\geq 1$. As in the proof of the previous lemma, let us denote by $(w_1 \ w_2 \ \dots \ w_\ell \ w_{\ell +1}\!\!=\!\!n+1)$ the cycle of $\sigma_{n+1}$ containing the element $n+1$, and by $z_1, \dots , z_\ell $ and $z_{n+1}$ the complex numbers of modulus one such that for all $k \in \{1, \dots , \ell \}$, $\widetilde{M}_{n+1} e_{w_k} = z_k e_{w_{k+1}}$ and $\widetilde{M}_{n+1} e_{n+1} = z_{n+1} e_{w_1}$ where $(e_1, \dots , e_{n+1})$ is the canonical basis of $\C^{n+1}$. The characteristic polynomials of $\widetilde{M}_{n+1}$ and $\widetilde{M}_n$ satisfy the equality
\begin{equation}\label{eq:charpolcustomer}
\chi_{\widetilde{M}_{n+1}} (X) = \left\{\begin{array}{ll}
  \frac{X^{\ell+1} - z_1 \dots z_\ell z_{n+1}}{X^\ell - z_1 \dots z_\ell z_{n+1}}	\chi_{\widetilde{M}_n} (X) & \text{if } \ell\geq 1 \\
  (X-z_{n+1}) \chi_{\widetilde{M}_n} (X) & \text{if } \ell =0
\end{array}\right. .
\end{equation}
Indeed, if $\ell=0$, then $\widetilde{M}_{n+1}$ can be written $\widetilde{M}_{n+1}=\mathrm{diag}(\widetilde{M}^{[n]}, z_{n+1}) = \mathrm{diag}(\widetilde{M}_n, z_{n+1}) $ by the previous lemma, so that $\chi_{\widetilde{M}_{n+1}} (X) = (X-z_{n+1}) \chi_{\widetilde{M}_n} (X) $. \\
Otherwise, there exists a permutation matrix $P$ of size $n+1$ which fixes the element $n+1$, such that $P\widetilde{M}_{n+1}P^{-1}$ and $P\mathrm{diag} (\widetilde{M}_n,1) P^{-1}$ are block diagonal matrices where:
\begin{itemize}
	\item All the blocks are of the form $(\alpha_1)$ or $\begin{pmatrix}
	0& \ldots & 0 & \alpha_j \\
	\alpha_1 & \ddots & & 0\\
	 & \ddots & \ddots & \vdots \\
	(0) & & \alpha_{j-1} & 0
\end{pmatrix}	$, with $\alpha_1, \dots, \alpha_j \in S^1$.
	\item If $k$ is the number of blocks of $P\widetilde{M}_{n+1}P^{-1}$, then $P\mathrm{diag} (\widetilde{M}_n,1) P^{-1}$ has exactly $k+1$ blocks (including the bottom-right $1$), and the $k-1$ first blocks of $P\widetilde{M}_{n+1}P^{-1}$ and $P\mathrm{diag} (\widetilde{M}_n,1) P^{-1}$ are equal. 
\end{itemize}
The last block of $P\widetilde{M}_{n+1}P^{-1}$ is $\begin{pmatrix}
	0& \ldots & 0 & z_{n+1} \\
	z_1 & \ddots & & 0\\
	 & \ddots & \ddots & \vdots \\
	(0) & & z_{\ell} & 0
\end{pmatrix}	$,
hence with the help of the previous lemma the penultimate block of $P\mathrm{diag} (\widetilde{M}_n,1) P^{-1}$ is
$\begin{pmatrix}
	0& \ldots & 0 & z_\ell z_{n+1} \\
	z_1 & \ddots & & 0\\
	 & \ddots & \ddots & \vdots \\
	(0) & & z_{\ell-1} & 0
\end{pmatrix}$, and we get
\begin{align*}
\chi_{\widetilde{M}_n} (X) &= \chi_{P^{[n]}\widetilde{M}_n (P^{-1})^{[n]}} (X) = \frac{\chi_{P \mathrm{diag}(\widetilde{M}_n,1) P^{-1}} (X) }{X-1} \\
	&= \frac{\chi_{P\widetilde{M}_{n+1}P^{-1}} (X)}{X^{\ell +1} - z_1 \dots z_\ell z_{n+1} } (X^\ell - z_1 \dots z_{\ell -1} (z_\ell z_{n+1})) 
\end{align*}
which gives \eqref{eq:charpolcustomer}. As a consequence, by analogy with the Chinese restaurant process (see \emph{e.g.} \cite{pitman2002combinatorial}), here the customers arrive one by one and choose a table according to its weight, regardless of the past, and when a new customer $n+1$ seats at a table (empty or not), it does not affect the element of $S^1 \setminus \{1\}$ corresponding to this table. Hence we can assign a $u_j$ to each table $j$, independently of $n$.
\end{proof}

\begin{definition}
We call \textbf{random modified permutation matrix} a random matrix $\widetilde{M}_n$ such that:
\begin{itemize}
	\item $\widetilde{M}_n$ corresponds to a random permutation $\sigma_n$ generated with the procedure of Proposition~\ref{prop:couplage} for a given distribution $p$ on $\nabla^\prime$.
	\item The non-zero entries of $\widetilde{M}_n$ are i.i.d random variables uniformly distributed on the unit circle.
\end{itemize}
\end{definition}

\begin{corollaire}\label{cor:couplage}
Let $(\sigma_n)_{n\geq 1}$ be a random virtual permutation, and let $(u_j)_{j\geq 1}$ be a sequence of i.i.d uniform variables on the unit circle, independent of $(\sigma_{n})_{n\geq 1}$. One can couple $((\sigma_n)_{n\geq 1} , (u_j)_{j\geq 1})$ with a random modified virtual permutation $(\widetilde{M}_n)_{n\geq 1}$ such that, for all $n\geq 1$, 
\begin{itemize}
	\item $\widetilde{M}_n$ is a random modified permutation matrix corresponding to $\sigma_{n}$.
	\item Denoting by $\ell_{n,j}$ the cycle-lengths of $\sigma_n$, then for all $j$ and $n$ such that $\ell_{n,j}>0$, $u_j$ is the product of the non-zero entries of $\widetilde{M}_n$ corresponding to the cycle $j$ of $\sigma_n$. 
\end{itemize}
\end{corollaire}
\begin{proof}
This immediately derives from Proposition~\ref{prop:couplage2} and the fact that the projection $M \mapsto N$ via $\mathrm{rank} ( M- \mathrm{diag} ( N, 1))$ is coherent with respect to the sequence of probability measures $(\mathcal{L}_n)$ defined for all $n$ as the law of a $n$-by-$n$ random modified permutation matrix. Indeed, if the non-zero entries of $\widetilde{M}_{n+1}$, say $z_1, z_2, \dots , z_{n+1}$, are i.i.d uniform variables on the unit circle, then the non-zero entries of $\widetilde{M}_n$, say $z_1^\prime, z_2^\prime, \dots , z_n^\prime$, satisfy the following rule: 
There exists $\pi \in \mathfrak{S}_{n+1}$ such that for all $j \in \{1, \dots ,n-1\}$, $z^\prime_j = z_{\pi (j)}$, and $z^\prime_n = z_{\pi (n)} z_{\pi (n+1)}$. 
Consequently $z_1^\prime, z_2^\prime, \dots , z_n^\prime$ are i.i.d uniform variables on the unit circle.
\end{proof}

\section{Proof of the main theorem}
\label{sec:quotient}

\subsection{Quotient of characteristic polynomials related to modified permutation matrices}
\label{sec:quotientMPM}
Consider a distribution with exponential decay $p$ on $\nabla^\prime$, giving a $r\in (0,1)$ as in \eqref{eq:expdecayrate}. Let $y=(y_1, y_2 , \dots )$ be a random vector following the distribution $p$.\\
Let $(\widetilde{M}_n)_{n\geq 1}$ be a sequence of modified random permutation matrices generated by the coupling given by Corollary~\ref{cor:couplage}. 

For all $n\in \N^*$ and $z \in \C$, we consider the characteristic polynomial of $\widetilde{M}_n$, defined by \eqref{eq:Ztilde}. As one is almost surely not a zero of $\widetilde{Z}_n$, the function $\widetilde{\xi}_n$ defined by \eqref{eq:xi1}
is an entire function.
 
Moreover, using the notations from \eqref{eq:ell} and \eqref{eq:yjn}, for all $n$ and $j$ such that $\ell_{n,j}>0$, $u_j$ is the product of the non-zero entries of $\widetilde{M}_n$ whom cycle is associated with the circle $\mathcal{C}_j$, so that $u_j$ does not depend on $n$. Hence, by Corollary~\ref{cor:couplage}, $\widetilde{\xi}_n$ can be reformulated with the help of the sequence $(u_j)_{j\geq 1}$ which is independent of the $y_k^{(n)}$, as
\begin{align}
\begin{split}
\widetilde{\xi}_n (z) &= \prod_{\substack{j\geq 1 \\ \ell_{n,j}>0}} \frac{\mathrm{e}^{2i\pi z y_j^{(n)}} - u_j}{1-u_j}  \\
	&= \prod_{\substack{j\geq 1 \\ \ell_{n,j}>0}} \left( 1 + \frac{1}{1-u_j} (\mathrm{e}^{2i\pi z y_j^{(n)}} -1 )    \right).
\end{split}
\end{align}

The next lemmas aim to handle the tail of the infinite product in the expression of $\widetilde{\xi}_n (z)$, in order to apply a dominated convergence theorem and get the pointwise convergence of $\widetilde{\xi}_n$. Moreover, they provide a bound of $\widetilde{\xi}_n$ uniformly on compact sets, allowing to conclude with Montel theorem.
\begin{lem}\label{lem1}
Let $\beta >2$. For all $k \in\N^*$, set $m_k := \min\limits_{1\leq j\leq k} \vert 1 - u_j \vert$. Then a.s. there exists a random number $C_1>0$ such that for all $k$, 
\begin{equation}
m_k > C_1 k^{-\beta}.
\end{equation}
\end{lem}
\begin{proof}
Let $A\in (0,1)$. Let $T$ be a random variable following the uniform distribution on $[0,1]$.  
\[\PP (\vert 1-\mathrm{e}^{2i\pi T} \vert \geq A ) \geq \PP (\sin (\pi T) \geq A ) \geq \PP (2 \min (T, 1-T) \geq A) = \PP (T \geq A) = 1- A.\]
Then for all $k$, 
\[\PP (m_k \leq A) \leq 1 - (1-A)^k \leq k A\]
using the mean value inequality. Thus, 
\[\sum_{k=1}^{+\infty} \PP (m_k \leq k^{-\beta}) \leq \sum_{k=1}^{+\infty} k^{1-\beta} < +\infty. \]
Applying Borel-Cantelli lemma we deduce that the number of $k$ such that $m_k \leq k^{-\beta}$ is a.s. finite, \emph{i.e.} a.s. there exists $k_0 \in \N^*$ such that for all $k> k_0$, $m_k>k^{-\beta}$. Finally,  $C_1:=\min\limits_{j\leq k_0} ( j^\beta \vert 1-u_j \vert )\wedge 1$ gives the claim.
\end{proof}

\begin{lem}\label{lem2}
For all $\rho \in (\sqrt{r} , 1)$, a.s. there exists a random number $C_2>0$, such that for all $j\geq 1$, 
\begin{equation}
s_j:=\sup_{n\geq 1} y_j^{(n)} \leq C_2 \rho^j.
\end{equation}
\end{lem}
\begin{proof}
Let $\rho\in (0,1)$. By the definition of $y_j^{(n)}$ in equation \eqref{eq:yjn}, we see that $y_j^{(n)}$ is the mean of $n$ i.i.d Bernoulli random variables $Z_{1,j} , \dots , Z_{n,j}$ of parameter $y_j$.
Fix $j$. \\
If $n \leq \rho^{-j}$, then it is easy to check that the events $\{ y_j^{(n)} \geq \rho^j \}$ and $\{ \exists k \in \{1, 2, \dots , n\} : \ Z_{k,j} =1\}$ are equal, hence
\[\E \left( \mathds{1}_{y_j^{(n)} \geq \rho^j} \ \vert \  y_j  \right) = \E \left( \mathds{1}_{\exists k \in \{1, 2, \dots , n\} : \ Z_{k,j} =1}  \ \vert \  y_j  \right) \leq n y_j \leq \frac{y_j}{\rho^j},\]
and then 
\[\sum_{n\leq \rho^{-j}} \PP_y ( y_j^{(n)} \geq \rho^j ) \leq \frac{y_j}{\rho^{2j}}.\]
If $n \geq \rho^{-j}$, then for any arbitrary $\lambda >0$ we have the Chernoff bound
\begin{align*}
\E \left( \mathds{1}_{y_j^{(n)} \geq \rho^j} \ \vert \  y_j  \right) &\leq \mathrm{e}^{-\lambda \rho^j} \E \left[ \mathrm{e}^{\frac{\lambda}{n} Z_{1,j} } \ \vert \ y_j  \right]^n \\
	&= \mathrm{e}^{-\lambda \rho^j} \left( 1 - y_j + y_j \mathrm{e}^{\frac{\lambda}{n}} \right)^n \\
	&\leq \mathrm{e}^{-\lambda \rho^j} \exp \left( n y_j \left( \mathrm{e}^{\frac{\lambda}{n}} -1 \right) \right).
\end{align*}
This inequality is optimized at point $\lambda = n \log (\rho^j / y_j)$, which gives
\[\E \left( \mathds{1}_{y_j^{(n)} \geq \rho^j} \ \vert \  y_j  \right) \leq \mathrm{e}^{-n \left(\rho^j \log \left( \frac{\rho^j}{y_j} \right) - \rho^j + y_j \right)} \leq  \mathrm{e}^{-n \rho^j \left( \log \left( \frac{\rho^j}{y_j} \right) - 1 \right) }. \]
By assumption of exponential decay \eqref{eq:expdecayrate}, conditionally on $y=(y_1, y_2, \dots )$, almost surely there exists a constant $C>0$ such that for all $j\geq 1$, $y_j\leq C r^j$. Thus, taking any arbitrary $r^\prime \in (r,1)$, for almost every $y$ there exists an integer $k$ such that for all $j\geq k$, $y_j\leq r^{\prime j}$. Fix a given $y$, $r^\prime$ and $k$.
Then, setting $\rho > \sqrt{r^\prime}$,
\[\log \left( \frac{\rho^j}{y_j} \right) \geq \log \left( \frac{\rho^j}{r^{\prime j}} \right) = j (\log \rho - \log r^\prime ) \geq 2 \]
for all $j$ sufficiently large and greater than $k$, say for all $j\geq m$ ($m$ dependent on $y$). Then for all $j\geq m$, 
\[\sum_{n\geq \rho^{-j}} \PP_y ( y_j^{(n)} \geq \rho^j ) \leq  \frac{\frac{y_j}{\rho^j}\mathrm{e}}{1- \mathrm{e}^{-\rho^j \left( \log \left( \frac{\rho^j}{y_j} \right) - 1 \right)}}  \leq \frac{\frac{y_j}{\rho^{j}}\mathrm{e}}{1-\mathrm{e}^{-\rho^j}} \leq (2\mathrm{e})  \frac{y_j}{\rho^{2j}}.\]
We deduce, for all $j\geq m$, 
\[\sum_{n=1}^{+\infty} \PP_y ( y_j^{(n)} \geq \rho^j ) \leq (1+2\mathrm{e})  \frac{y_j}{\rho^{2j}} \leq (1+2\mathrm{e}) \left(\frac{r^\prime}{\rho^{2}}\right)^j \]
and consequently, we get
\[\E \left[ \sum_{j=m}^{+\infty} \sum_{n=1}^{+\infty} \mathds{1}_{y_j^{(n)} \geq \rho^j} \ \vert \  y \right] < +\infty \]
which implies
\[\PP_y \left(\sum_{j=m}^{+\infty} \sum_{n=1}^{+\infty} \mathds{1}_{y_j^{(n)} \geq \rho^j} < +\infty  \right) =1\]
for almost every $y$. Finally, taking the expectation we get
\[\PP \left(\sum_{j=m}^{+\infty} \sum_{n=1}^{+\infty} \mathds{1}_{y_j^{(n)} \geq \rho^j} < +\infty \right) =1.\] In other words, the number of couples $(j,n)$ such that $j\geq m$ and $y_j^{(n)} \geq \rho^j$ is almost surely finite, which gives that the number of $j$ such that $\sup\limits_{n\in \N^*} y_j^{(n)} \geq \rho^j$ is almost surely finite.
\end{proof}

\begin{lem}
With the same notation as above, a.s.,
\begin{equation}\label{eq:yjuj}
C_3:=\sum_{j=1}^{+\infty} \frac{y_j}{\vert 1-u_j \vert} < +\infty
\end{equation}
and 
\begin{equation}
C_4:=\sum_{j=1}^{+\infty} \frac{s_j}{\vert 1-u_j \vert} < +\infty .
\end{equation}
\end{lem}
\begin{proof}
Straightforward consequence of Lemma~\ref{lem1} and Lemma~\ref{lem2}.
\end{proof}

Now, we are able to prove the first point of Theorem~\ref{thm:convergence}:
\begin{proof}[Proof of Theorem~\ref{thm:convergence} $(i)$]
We begin to show the pointwise convergence. Let $z\in \C$. The idea of the proof consists in splitting the product in the expression of $\widetilde{\xi}_n (z)$ as follows:
\[\widetilde{\xi}_n (z) = \prod_{j=1}^{j_0} \frac{\mathrm{e}^{2i\pi z y_j^{(n)}} - u_j}{1-u_j} \prod_{j=j_0+1}^{+\infty} \frac{\mathrm{e}^{2i\pi z y_j^{(n)}} - u_j}{1-u_j} \]
where $j_0$ is an integer depending on $\vert z \vert$ and on random numbers $C_1$ and $C_2$, chosen in such a way that for all $j>j_0$, 
\[\left\vert  \frac{1}{1-u_j} (\mathrm{e}^{2i\pi z y_j^{(n)}} -1   ) \right\vert  \leq \frac{9}{10} < 1. \] 
Indeed we can chose such a $j_0$ as for all $j$, using Lemma~\ref{lem1} with $\beta=3$ and Lemma~\ref{lem2} we have
\begin{align*}
\left\vert  \frac{1}{1-u_j} (\mathrm{e}^{2i\pi z y_j^{(n)}} -1   ) \right\vert  &\leq \frac{1}{C_1 j^{-3}} 2\pi \vert z \vert y_j^{(n)} \exp (2\pi \vert z \vert y_j^{(n)} ) \\
	&\leq \frac{1}{C_1 j^{-3}} 2\pi \vert z \vert C_2 \rho^j \exp (2\pi \vert z \vert C_2 \rho^j )
\end{align*}
and then it suffices to take $j$ large enough such that $\max ( \frac{1}{C_1 j^{-3}} 2\pi \vert z \vert C_2 \rho^j , 2\pi \vert z \vert C_2 \rho^j ) \leq \frac{1}{2}$, which provides a bound lower than $\frac{1}{2} \mathrm{e}^{1/2} \approx 0.82 \leq \frac{9}{10}$. 
Thus, we can apply the logarithm to the product of terms for $j>j_0$ in the expression of $\widetilde{\xi}_n$, and furthermore it is straightforward to check that for all $w\in \C$ such that $\vert w \vert \leq \frac{9}{10}$ we have 
\[\vert \log (1+w) \vert = \left\vert \sum_{k=1}^{+\infty} \frac{(-1)^{k+1}}{k} w^k  \right\vert \leq \sum_{k=1}^{+\infty} \frac{\vert w \vert^k}{k} \leq \frac{\vert w \vert}{1- \vert w \vert } \leq 10 \vert w \vert.  \] 
Consequently for all $j>j_0$,
\begin{align*}
\left\vert \log \left( 1 + \frac{1}{1-u_j} (\mathrm{e}^{2i\pi z y_j^{(n)}} -1   ) \right) \right\vert &\leq 10 \left\vert  \frac{1}{1-u_j} (\mathrm{e}^{2i\pi z y_j^{(n)}} -1   ) \right\vert  \\
	&\leq \frac{10}{\vert 1- u_j \vert} \sum_{k=1}^{+\infty}  \frac{\vert 2i\pi z y_j^{(n)}  \vert^k}{k!} \\
	&\leq \frac{10}{\vert 1- u_j \vert} y_j^{(n)} (\mathrm{e}^{2\pi \vert z\vert } -1) \\
	&\leq \frac{10 \mathrm{e}^{2\pi \vert z\vert } C_2}{C_1} \rho^j j^3
\end{align*}
which is summable in $j$. Moreover, as $y_j^{(n)}$ converges a.s. to $y_j$, then by continuity $\log \left( 1 + \frac{1}{1-u_j} (\mathrm{e}^{2i\pi z y_j^{(n)}} -1   ) \right)$ converges to $\log \left( 1 + \frac{1}{1-u_j} (\mathrm{e}^{2i\pi z y_j} -1   ) \right)$. Hence by dominated convergence,
\[\widetilde{\xi}_n (z) \underset{n\to\infty}{\longrightarrow} \widetilde{\xi}_\infty (z).\]
This holds true for every $z\in \C$ so we get the pointwise convergence of $\widetilde{\xi}_n$ to $\widetilde{\xi}_\infty$ on $\C$.   \\
Now, consider an arbitrary compact set of $\C$ included in $\{ z \in \C : \vert z \vert \leq K \}$ for a certain fixed $K>0$. For all $z$ in this compact set and all $n\geq 1$,
\begin{align*}
\vert \widetilde{\xi}_n (z) \vert &\leq \prod_{j=1}^{+\infty} \left( 1 + \frac{1}{\vert 1-u_j \vert} \vert \mathrm{e}^{2i\pi z y_j^{(n)}} -1 \vert \right) \\
	&\leq \prod_{j=1}^{+\infty} \left( 1 + \frac{y_j^{(n)}}{\vert 1-u_j \vert}  (\mathrm{e}^{2 \pi K} -1)  \right) \\
	&\leq \exp \left( \sum_{j=1}^{+\infty}  \frac{y_j^{(n)}}{\vert 1-u_j \vert}  (\mathrm{e}^{2 \pi K} -1)   \right) \\
	&\leq \exp \left( (\mathrm{e}^{2 \pi K} -1)  \sum_{j=1}^{+\infty}  \frac{s_j}{\vert 1-u_j \vert}  \right) \\
	&= \exp \left( C_4 (\mathrm{e}^{2 \pi K} -1) \right).
\end{align*}
We deduce by Montel theorem the uniform convergence of $\widetilde{\xi}_n$ to $\widetilde{\xi}_\infty$ on all compact sets. \\
Finally, the functions $g_j (z):= \frac{1}{1-u_j} \left( \mathrm{e}^{2i\pi z y_j } -1  \right)$ are holomorphic on $\C$, and for all $z$ in any arbitrary compact subset of $\C$, $\vert g_j (z)\vert $ is uniformly bounded by a constant depending on the compact set times $\frac{y_j}{\vert 1-u_j \vert }$, which is summable in $j$ (see \eqref{eq:yjuj}). Hence $\sum g_j$ is normally convergent on compact sets, and it follows that the infinite product $\prod (1 + g_j)$ is uniformly convergent on compact sets. One deduces that $\widetilde{\xi}_\infty$ is an entire function.
\end{proof}

\subsection{Quotient of characteristic polynomials related to permutation matrices (without modification)}
\label{sec:quotientPM}

Here, the problem of finding a suitable normalization of the characteristic polynomial in order to have a non-trivial limiting function is more difficult than the previous one. We precise below the nature of this problem and defend our choice of normalization. 

Consider a distribution with exponential decay $p$ on $\nabla^\prime$, giving a $r\in (0,1)$ as in \eqref{eq:expdecayrate}, and let $y=(y_1, y_2 , \dots )$ be a random vector following the distribution $p$.\\
Let $(M_n)_{n\geq 1}$ be a sequence of random permutation matrices generated by the coupling described in Section~\ref{sec:gen} with respect to $y$. For all $n\in \N^*$ and $z \in \C$, we consider the characteristic polynomial of $M_n$ defined by \eqref{eq:Z}. 

Contrarily to random permutation matrices with modification (the $\widetilde{M}_n$ defined in the previous subsection), for every $n$, there are some points $z$ on the unit circle such that $\PP (Z_n (z)=0)>0$. For instance, for all $n$, the characteristic polynomial of $M_n$ evaluated at $z=1$ is almost surely zero (since each $j$-cycle of the associated permutation corresponds to eigenvalues which are exactly the $j-$th roots of unity). Thus, the function $\widetilde{\xi}_n$ of the previous section, replacing $\widetilde{Z}_n$ by $Z_n$, \emph{i.e.}
\[\widetilde{\xi}_n (z) = \frac{Z_n (\mathrm{e}^{2i \pi z/n})}{Z_n (1)}\]
is not well-defined on the whole complex plane here. Based on the fact that all eigenvalues of permutation matrices are roots of unity, then, for every irrational number $\alpha$, $z=\mathrm{e}^{2i\pi\alpha}$ is almost surely not a zero of $Z_n$ for all $n$. Let $\alpha$ be an irrational number between $0$ and $1$. It is natural to shift the random process of eigenangles by $2\pi \alpha$, and consider the function $\xi_{n,\alpha}$ defined by \eqref{eq:xialpha1}.

As $\alpha$ is irrational, then $\xi_{n,\alpha}$ is an entire function, and can be written as follows:
\begin{align}
\begin{split}
\xi_{n,\alpha } &= \prod_{\substack{j\geq 1 \\ \ell_{n,j}>0}} \frac{\mathrm{e}^{2i\pi \left(\frac{z}{n} + \alpha \right)\ell_{n,j}} -1}{\mathrm{e}^{2i\pi \alpha \ell_{n,j}} -1} \\
 &=\prod_{\substack{j\geq 1 \\ \ell_{n,j}>0}}  \left( 1 + \frac{\mathrm{e}^{2i\pi \alpha \ell_{n,j}}}{\mathrm{e}^{2i\pi \alpha \ell_{n,j}} -1} \left( \mathrm{e}^{2i\pi z y_j^{(n)}} -1 \right)   \right).
\end{split}
\end{align}
Heuristically, the idea of considering $\alpha$ of finite type in Theorem~\ref{thm:convergence} is to get the denominators in the expression of $\xi_{n,\alpha }$ not too close from $0$ when $n$ become large, by comparison to the factor $\mathrm{e}^{2i\pi z y_j^{(n)}} -1$ in the numerator. In the sequel we prove the result as follows: first, we show the convergence in law for finite products in the topology of pointwise convergence, and then, we handle the remaining infinite product in order to get the convergence of $\xi_{n,\alpha}$ to $\widetilde{\xi}_\infty$ in this topology. Finally we use a criterion to extend the result to the topology of uniform convergence on compact sets.

\begin{lem}
For all $k\in \N^*$, conditionally on $y$,
\[(\ell_{n,1} , \ell_{n,2} , \dots , \ell_{n,k} , n- \ell_{n,1} - \ell_{n,2}- \dots - \ell_{n,k}) \overset{\text{d}}{=} \mathcal{M} (n, y_1 , y_2 , \dots ,y_k, 1 - y_1 - y_2 - \dots - y_k),\]
where $\mathcal{M} (n,q_1,q_2,\dots , q_m)$ denotes a multinomial random variable of parameters $q_1,q_2,\dots ,q_m$.
\end{lem}
\begin{proof}
Direct consequence of \eqref{eq:ell}.
\end{proof}

\begin{lem}
For all $k\in \N^*$, conditionally on $y$,
\[\left(\{\alpha \ell_{n,1} \}, \{ \alpha \ell_{n,2} \} , \dots , \{ \alpha \ell_{n,k} \}  , y_1^{(n)} , \dots ,  y_k^{(n)}  \right) \overset{d}{\underset{n\to \infty}{\longrightarrow}} (\Phi_1, \Phi_2,\dots , \Phi_k, y_1, y_2 , \dots , y_k) \]
where the $\Phi_j$ are i.i.d random variables uniformly distributed on $[0,1]$, independent of the $y_j$. 
\end{lem}
\begin{proof}
Let $j_1, \dots , j_k$ be integers, and let $\lambda_1, \dots , \lambda_k$ be real numbers. Conditionally on $y$, the Fourier transform of the distribution of $\left(\{\alpha \ell_{n,1} \}, \{ \alpha \ell_{n,2} \} , \dots , \{ \alpha \ell_{n,k} \}  , y_1^{(n)} , \dots ,  y_k^{(n)}  \right)$ reads
\begin{align*}
&\E \left. \left( \mathrm{e}^{2i\pi (j_1 \{\alpha \ell_{n,1} \} + \dots + j_k \{ \alpha \ell_{n,k} \}) + i \left( \lambda_1 \frac{\ell_{n,1}}{n} + \dots + \lambda_k \frac{\ell_{n,k}}{n} \right)} \right\vert y \right) \\
	&\qquad = \E \left. \left( \mathrm{e}^{2i\pi \left( j_1 \alpha + \frac{\lambda_1}{2\pi n} \right)  \ell_{n,1} + \dots + 2i \pi \left( j_k \alpha + \frac{\lambda_k}{2\pi n} \right) \ell_{n,k}  + 0\times ( n-\ell_{n,1}-\dots -\ell_{n,k} ))} \right\vert y \right) \\
	&\qquad = \sum_{\substack{(\ell_1 , \dots , \ell_{k+1}) \in \N^{k+1} \\ \ell_1 + \dots + \ell_{k+1} = n }} \frac{n!}{\ell_1 ! \dots \ell_{k+1} !} \left( y_1 \mathrm{e}^{2i\pi \left( j_1 \alpha + \frac{\lambda_1}{2\pi n} \right)}\right)^{\ell_1} \dots \left( y_k \mathrm{e}^{2i\pi \left( j_k \alpha + \frac{\lambda_k}{2\pi n} \right)}\right)^{\ell_k} (1-y_1- \dots - y_k)^{\ell_{k+1}} \\
	&\qquad = \left(  y_1 \mathrm{e}^{2i\pi \left( j_1 \alpha + \frac{\lambda_1}{2\pi n} \right)} + \dots +  y_k \mathrm{e}^{2i\pi \left( j_k \alpha + \frac{\lambda_k}{2\pi n} \right)} + (1-y_1 - \dots - y_k) \right)^n. 
\end{align*}
If $j_1 = j_2 = \dots = j_k = 0$, then this quantity converges to $\mathrm{e}^{i (\lambda_1 y_1 + \dots \lambda_k y_k)}$. \\
Otherwise, there exists $m \in [\![ 1, k ]\!]$ such that $j_m \neq 0$. Since $\alpha$ is irrational, $\Vert j_m \alpha \Vert \neq 0$. Hence there exists $N \in \N$ such that for all $n\geq N$, $\left\vert \frac{\lambda_m}{2\pi n} \right\vert \leq \frac{\Vert j_m \alpha \Vert}{2} < \frac{1}{2}$. Then for all $n\geq N$, \[\left\Vert j_m \alpha + \frac{\lambda_m}{2\pi n} \right\Vert \geq \Vert j_m \alpha \Vert - \left\vert \frac{\lambda_m}{2\pi n} \right\vert \geq  \frac{\Vert j_m \alpha \Vert}{2}.\]
Consequently, $y_1 \mathrm{e}^{2i\pi \left( j_1 \alpha + \frac{\lambda_1}{2\pi n} \right)} + \dots  y_k \mathrm{e}^{2i\pi \left( j_k \alpha + \frac{\lambda_k}{2\pi n} \right)} + (1-y_1 - \dots - y_k)$ is a convex combination with a.s. positive coefficients of points located on the unit circle, whose the distance between two points (namely the point $1$ and the point $\mathrm{e}^{2i\pi \left( j_m \alpha + \frac{\lambda_m}{2\pi n} \right)}$) is bounded from below by a positive number which is independent of $n$. Thus $y_1 \mathrm{e}^{2i\pi \left( j_1 \alpha + \frac{\lambda_1}{2\pi n} \right)} + \dots  y_k \mathrm{e}^{2i\pi \left( j_k \alpha + \frac{\lambda_k}{2\pi n} \right)} + (1-y_1 - \dots - y_k)$ is bounded, uniformly in $n$, by a quantity strictly smaller than one, and finally the Fourier transform goes to $0$, which gives the claim. Note that it is enough to consider only integer values for the $j_m$ because the fractional part function takes values in $[0,1)$, which can be identified to the quotient $\R/\Z$. As the representations of $\R/\Z$ are the functions $x\mapsto \mathrm{e}^{2i\pi j x}$ for $j\in \Z$, the law of every random variable on $[0,1)$ is fully determined by its Fourier transform at integers.
\end{proof}

\begin{prop}\label{prop:convpart1}
For all $k\in \N^*$, conditionally on $y$,
\[  \prod_{\substack{1\leq j\leq k \\ \ell_{n,j}>0}} \left( 1  + \frac{\mathrm{e}^{2i\pi \alpha \ell_{n,j}}}{\mathrm{e}^{2i\pi \alpha \ell_{n,j}} -1 } \left(  \mathrm{e}^{2i\pi z y_j^{(n)}} -1  \right) \right)  \underset{n\to \infty}{\Longrightarrow}  \prod_{\substack{1\leq j\leq k }}  \left( 1  + \frac{\mathrm{e}^{2i\pi \Phi_j}}{\mathrm{e}^{2i\pi \Phi_j} -1 } \left(  \mathrm{e}^{2i\pi z y_j} -1  \right) \right).\]
\end{prop}
\begin{proof}
Let $k\in \N^*$. The map $(x_1 , x_2 , \dots ,x_{2k})  \mapsto \left( z \mapsto \prod_{j=1}^k \left( 1  + \frac{\mathrm{e}^{2i\pi x_j}}{\mathrm{e}^{2i\pi x_j} -1 } \left(  \mathrm{e}^{2i\pi z x_{k+j} } -1\right) \right) \right)$ is defined and continuous on $(\R \setminus \Z)^k\times \R^k$. Moreover, $\PP ( \exists j \in [\![ 1, k ]\!], \  \Phi_j \in \Z) = 0$. Thus, the convergence in distribution directly follows from the previous lemma and the continuous mapping theorem.
\end{proof}

\begin{lem}
Let $p\in [0,1/2]$ and let $B_n$ be a random variable following the binomial distribution of parameters $n,p$, for any $n \in \N^*$. Then $(B_n)_{n\geq 1}$ satisfies the following properties:
\begin{enumerate}
	\item[(i)] For all $k \in \N^*$, 
	\[\sup_{n \in \N^*} \PP (B_n = k) \ll \frac{1}{\sqrt{k}}.\]
	\item[(ii)] Let $1\leq m\leq n$. Let $E$ be an ensemble of positive integers such that:
	\begin{itemize}
		\item For all distinct $j_1,j_2 \in E$, $\vert j_1 - j_2 \vert \geq m$.
		\item For all $j \in E$, $j\geq m$.
	\end{itemize} Then
	\[\PP (B_n \in E) = \mathcal{O} \left(\frac{1}{\sqrt{m}}\right),\] 
	where the $\mathcal{O} (\frac{1}{\sqrt{m}})$ is independent of $n$ and $p$.
\end{enumerate}
\end{lem}
\begin{proof}
\underline{Proof of $(i)$}:
Fix $k\in \N^*$. Using Stirling's formula and the fact that $x^k (1-x)^{n-k}$ is maximal for $x=\frac{k}{n}$, 
\begin{align*}
\PP (B_n=k) = \binom{n}{k} p^k (1-p)^{n-k} &\ll \frac{\left(\frac{n}{e}\right)^{n+1/2}}{\left(\frac{n-k}{e}\right)^{n-k+1/2} k!} \left(\frac{k}{n}\right)^k \left(1- \frac{k}{n}\right)^{n-k} \\
	&= \frac{1}{k!} \left(\frac{k}{e}\right)^k \sqrt{\frac{n}{n-k}} \\
	&\ll \sqrt{\frac{n}{k(n-k)}}.
\end{align*}  
Hence, if $1\leq k \leq \frac{3}{4}n$ we have $\PP (B_n=k)\ll \frac{1}{\sqrt{k}}$. If $\frac{3}{4} n \leq k \leq n$, as $x \to x^k (1-x)^{n-k}$ is increasing on $[0 , \frac{k}{n}]$ and $p\leq \frac{1}{2}\leq \frac{k}{n}$ by hypothesis,   
\[\PP (B_n = k) \leq \binom{n}{\frac{n}{4}} \left(\frac{1}{2}\right)^n \ll \left( \frac{\frac{1}{2}}{\left(\frac{1}{4}\right)^{1/4} \left(\frac{3}{4}\right)^{3/4} }  \right)^n = \left( \frac{2}{3^{3/4}}  \right)^n \leq \left( \frac{2}{3^{3/4}}  \right)^k \ll \frac{1}{\sqrt{k}}.\] 

\underline{Proof of $(ii)$}:
Fix $n$ and denote $f_k:= \binom{n}{k} p^k (1-p)^{n-k}$. For all $0\leq k<n$, it is easy to check that 
\[\frac{f_{k+1}}{f_k} < 1 \Longleftrightarrow k > (n+1)p-1,\]
so the sequence $(f_k)_{k \in [\![ 0 , n ]\!]}$ is increasing for $k\leq k_0$, and decreasing for $k\geq k_0 +1$, where $k_0:= \lfloor (n+1)p-1 \rfloor$. \\
If $m\geq k_0+1$, then 
\[\PP (B_n \in E) = \sum_{\substack{k\in E \\ k\leq n }} f_k \leq \sum_{r=0}^{\left\lfloor \frac{n}{m} \right\rfloor -1  } f_{m + mr} \leq f_{m} + \frac{1}{m} \sum_{j=m}^{ m\left\lfloor \frac{n}{m} \right\rfloor -1} f_j \leq f_{m} + \frac{1}{m} \sum_{j=0}^{n} f_j = f_m + \frac{1}{m},\]
and by $(i)$ we deduce $\PP (B_n \in E) \ll \frac{1}{\sqrt{m}}$.\\
If $m\leq k_0$, then we look separately at the increasing part and the decreasing part: on the one hand,
\[\sum_{\substack{k\in E \\ k\leq k_0 }} f_k \leq \sum_{r=0}^{\left\lfloor \frac{k_0}{m} \right\rfloor } f_{k_0 -mr} \leq f_{k_0} + \frac{1}{m} \sum_{j=k_0 - m \left\lfloor \frac{k_0}{m} \right\rfloor  +1}^{k_0} f_j \leq f_{k_0} + \frac{1}{m} \sum_{j=0}^{k_0} f_j, \]
and on the other hand, 
\[\sum_{\substack{k\in E \\ k_0 < k \leq n}} f_k \leq \sum_{r=0}^{\left\lfloor \frac{n-(k_0 + 1)}{m} \right\rfloor  } f_{k_0 + 1 + mr} \leq f_{k_0 + 1} + \frac{1}{m} \sum_{j=k_0 + 1}^{k_0 + m\left\lfloor \frac{n-(k_0 + 1)}{m} \right\rfloor } f_j \leq f_{k_0 + 1} + \frac{1}{m} \sum_{j=k_0 + 1}^{n} f_j .\]
Thus,
\[\PP (B_n \in E) \leq f_{k_0} + f_{k_0 + 1} + \frac{1}{m} \sum_{j=0}^{n} f_j =  f_{k_0} + f_{k_0 + 1} + \frac{1}{m},\]
and by $(i)$ we deduce $\PP (B_n \in E) \ll \frac{1}{\sqrt{k_0}} + \frac{1}{m} \ll \frac{1}{\sqrt{m}}$. 
\end{proof}

\begin{lem}\label{lem:distanceInteger}
Assume $\alpha$ is an irrational number of finite type $\eta \geq 1$. Let $(a_j)_{j\geq 1}$ be any arbitrary sequence of positive real numbers. For all $j$, let $E_j:= \{ \ell \in \N^* : \ \Vert \alpha \ell \Vert \leq a_j \}$. Then for all $\nu > \eta$, there exists a number $c_\nu$ depending only on $\nu$ and $a_1$, such that for all $j$,
\[\sup\limits_{n\in \N^*} \PP_y ( \ell_{n,j}  \in E_j ) \leq c_\nu  a_j^{\frac{1}{2\nu}}.\]
\end{lem}
\begin{proof}
Let $\nu >\eta$. From \eqref{eq:type}, 
\[\exists C >0 , \ \forall \ell \in \N^*, \ \Vert \alpha \ell \Vert \geq \frac{C}{\ell^\nu}. \]
Let $(a_j)_{j \geq 1}$ be a sequence of positive real numbers. For all $j \in \N^*$, it is easy to check that the set $E_j:= \{ \ell \in \N^* : \ \Vert \alpha \ell \Vert \leq a_j \}$ satisfies
\begin{itemize}
	\item $\forall \ell \in E_j$, $\ell \geq C^{1/\nu} a_j^{-1/\nu }$.
	\item $\forall \ell_1 \neq \ell_2 \in E_j$, $\vert \ell_1 - \ell_2 \vert \geq \left(\frac{C}{2} \right)^{1/\nu} a_j^{-1/\nu }$.
\end{itemize}   
Moreover, by construction, conditionally on $y$ each random variable $\ell_{n,j}$ follows a binomial distribution of parameters $n , y_j$. Since almost surely the $y_j$ sum to one and $y_1$ is the largest, then $y_j\leq \frac{1}{2}$ for all $j\neq 1$, hence the previous lemma applies with $m= \left( \frac{C}{2} \right)^{1/\nu} a_j^{-1/\nu }$ and gives for all $j\geq 2$,
\[\PP_y ( \ell_{n,j}  \in E_j) = \mathcal{O}\left(  a_j^{\frac{1}{2\nu}} \right), \]
where the $\mathcal{O}\left(  a_j^{\frac{1}{2\nu}} \right)$ is independent of $n$ and of the $y_j$, $j \geq  2$. In other words, there exists a number $b_\nu$ depending only on $\nu$ and $y$, such that for all $j \geq 2$, 
\[\sup\limits_{n\in \N^*} \PP_y ( \ell_{n,j}  \in E_j ) \leq b_\nu  a_j^{\frac{1}{2\nu}}.\]
Finally we get the result taking $c_\nu:= \max \left( b_\nu , a_1^{-\frac{1}{2\nu}} \right)$.
\end{proof}

\begin{prop}\label{prop:convpart2}
Conditionally on $y$, for all $\varepsilon >0$ and for all compact subsets $K$ of $\C$, 
\[\sup_{n \in \N^*} \PP_y\left( \sup_{z\in K} \left\vert \prod_{\substack{j>k \\ \ell_{n,j}>0}} \left( 1  + \frac{\mathrm{e}^{2i\pi \alpha \ell_{n,j}}}{\mathrm{e}^{2i\pi \alpha \ell_{n,j}} -1 } \left(  \mathrm{e}^{2i\pi z y_j^{(n)}} -1  \right) \right) -1\right\vert \geq \varepsilon \right) \underset{k\to \infty}{\longrightarrow} 0.\]
\end{prop}
\begin{proof}
Let $\varepsilon >0$ and let $K$ be a compact subset of $\C$. It suffices to show that conditionally on $y$,  
\[\sup_{n \in \N^*} \PP_y \left( \sum_{\substack{j>k \\ \ell_{n,j}>0}} \frac{y_j^{(n)}}{\Vert \alpha \ell_{n,j} \Vert} \geq \varepsilon\right) \underset{k\to \infty}{\longrightarrow} 0,\]
Indeed, for all $n,j$ such that $\ell_{n,j}>0$,
\begin{equation}\label{eq:yjellbound}
\left\vert \frac{\mathrm{e}^{2i\pi \alpha \ell_{n,j}}}{\mathrm{e}^{2i\pi \alpha \ell_{n,j}} -1 } \left(  \mathrm{e}^{2i\pi z y_j^{(n)}} -1  \right) \right\vert \leq \frac{1}{\min ( \{ \alpha \ell_{n,j} \} , 1- \{\alpha \ell_{n,j}\} )} C_K y_j^{(n)} = \frac{C_K y_j^{(n)}}{\Vert \alpha \ell_{n,j} \Vert} 
\end{equation}
where $C_K$ is a constant number that only depends on $K$, and moreover for all sequences of functions $g_j$, 
\begin{align*}
\left\vert \prod_j (1+g_j (z)) -1 \right\vert &\leq \prod_j (1+ \vert g_j (z) \vert ) -1 \\\
	&\leq \prod_j (1+ \sup_K \vert g_j \vert ) - 1 \\
	&\leq \exp \left( \sum_j \sup_K \vert g_j \vert \right) - 1.
\end{align*}  
Let $n,k\geq 1$. Let $s,\rho \in (0,1)$ such that $r<\rho <s<1$. We have, denoting $A:=\left\{ \sum\limits_{\substack{j>k \\ \ell_{n,j}>0}} \frac{y_j^{(n)}}{\Vert \alpha \ell_{n,j} \Vert} \geq \varepsilon \right\}$ and $B_j:=\left\{ \Vert \alpha \ell_{n,j} \Vert \leq s^j , \ \ell_{n,j}>0 \right\}$,
\begin{align*}
\PP_y (A) &= \PP_y (A \cap B_{k+1}) +  \PP_y (A \cap B_{k+2}) + \dots +  \PP_y (A \cap ( B_{k+1} \cup B_{k+2} \cup \dots )^\complement ) \\
	&\leq \left(\sum_{j>k} \PP_y (B_j) \right) + \PP_y (A \cap ( B_{k+1} \cup B_{k+2} \cup \dots )^\complement ).
\end{align*}
On the one hand, for all $j$, $\PP_y (B_j) = \PP_y (\ell_{n,j} \in E_j)$ with $E_j:= \{ \ell \in \N^* : \ \Vert \alpha \ell \Vert \leq s^j \}$, then it follows from Lemma~\ref{lem:distanceInteger} that $\PP_y (B_j) = \mathcal{O}\left( s^{\frac{j}{2\nu}}\right)$ independent of $n$, hence 
\[\sup_{n\in \N^*} \sum_{j>k} \PP_y (B_j) \underset{k\to \infty}{\longrightarrow} 0. \]
On the other hand, 
\begin{align*}
\PP_y (A \cap ( B_{k+1} \cup B_{k+2} \cup \dots )^\complement ) &= \PP_y (A \cap (\forall j>k , \ \Vert \alpha \ell_{n,j} \Vert > s^j \ \text{ or } \ell_{n,j}=0 ) ) \\
	&\leq \PP_y \left(\widetilde{A} \right),
\end{align*}
with $\widetilde{A}:=\left\{ \sum\limits_{\substack{j>k \\ \ell_{n,j}>0}} \frac{y_j^{(n)}}{s^j} \geq \varepsilon \right\}$, and furthermore denoting $\widetilde{B}_j:=\left\{ y_j^{(n)} \geq \rho^j \right\}$ it comes
\begin{align*}
\PP_y (\widetilde{A}) &\leq \left( \sum_{j>k} \PP_y (\widetilde{B}_j) \right) + \PP_y (\widetilde{A} \cap (\forall j > k , \ y_j^{(n)} < \rho^j ) ) \\
	&\leq \left(\sum_{j>k} \frac{\E (y_j^{(n)} \ \vert \ y)}{\rho^j} \right) + \PP_y \left( \sum_{j>k} \frac{\rho^j}{s^j} \geq \varepsilon \right) \\
	&= \left(\sum_{j>k} \frac{y_j}{\rho^j} \right) + \mathds{1}_{\sum\limits_{j>k} \left(\frac{\rho}{s}\right)^j \geq \varepsilon}
\end{align*}
independent of $n$, hence 
\[\sup_{n\in \N^*} \PP_y (A \cap ( B_{k+1} \cup B_{k+2} \cup \dots )^\complement ) \underset{k\to \infty}{\longrightarrow} 0. \]
Consequently,
\[\sup_{n \in \N^*} \PP_y \left( \sum_{\substack{j>k \\ \ell_{n,j}>0}} \frac{y_j^{(n)}}{\Vert \alpha \ell_{n,j} \Vert} \geq \varepsilon  \right) \underset{k\to \infty}{\longrightarrow} 0. \]
\end{proof}

\begin{prop}\label{prop:convpart3}
Let $K$ be a compact subset of $\C$. Conditionally on $y$, the sequence $(\sup_K \vert \xi_{n,\alpha }\vert)_{n\geq 1}$ is tight.
\end{prop}

\begin{proof}
It follows from Proposition~\ref{prop:convpart2} that for all $\delta >0$,
\[ \PP_y\left( \sup_{z\in K} \left\vert \prod_{\substack{j>k \\ \ell_{n,j}>0}} \left( 1  + \frac{\mathrm{e}^{2i\pi \alpha \ell_{n,j}}}{\mathrm{e}^{2i\pi \alpha \ell_{n,j}} -1 } \left(  \mathrm{e}^{2i\pi z y_j^{(n)}} -1  \right) \right) \right\vert < 1.01 \right) \geq 1-\delta \]
for all $k$ large enough, say $k \geq k_0 (\delta)$ independent of $n$. Fix $\delta$. If we show that the sequence of the supremum on $K$ of the product for $j\leq k_0 (\delta )$ is tight, then for all $A$ large enough (depending on $\delta$) we will have
 \[ \PP_y\left( \sup_{z\in K} \left\vert \prod_{\substack{j\leq k_0 (\delta ) \\ \ell_{n,j}>0}} \left( 1  + \frac{\mathrm{e}^{2i\pi \alpha \ell_{n,j}}}{\mathrm{e}^{2i\pi \alpha \ell_{n,j}} -1 } \left(  \mathrm{e}^{2i\pi z y_j^{(n)}} -1  \right) \right) \right\vert < A \right) \geq 1-\delta \]
so that 
\[\PP_y (\sup_K \vert \xi_{n,\alpha }\vert < 1.01 A) \geq 1-2\delta ,\]
which gives the desired tightness. From \eqref{eq:yjellbound}, the tightness of the supremum for finite products are implied by the tightness of 
\[\sum_{\substack{j\leq k \\ \ell_{n,j}>0}} \frac{y_j^{(n)}}{\Vert \alpha \ell_{n,j} \Vert}\]
for every fixed $k$. 
We have, for all $A,M>0$, and all $s,\rho$ such that $r<\rho<s<1$, 
\begin{align}\label{eq:tightsum}
\begin{split}
 \PP_y \left( \sum_j \frac{y_j^{(n)}}{\Vert \alpha \ell_{n,j} \Vert} \geq M \right) &\leq  \sum_j \PP_y \left( \Vert \alpha \ell_{n,j} \Vert \leq \frac{1}{A}s^j , \ell_{n,j} >0 \right) \\
 &\quad + \sum_j \PP_y (y_j^{(n)} \geq A \rho^j ) + \PP_y \left( \sum_j \frac{A \rho^j}{\frac{1}{A} s^j} \geq M \right)
\end{split}
\end{align}
Applying Lemma~\ref{lem:distanceInteger} with $a_j= \frac{1}{A}s^j$, the first right-hand-side term of \eqref{eq:tightsum} is dominated by $A^{-\frac{1}{2\nu}}$ for any $\nu$ greater than the type of $\alpha$, independently of $n$. Using Markov's inequality and the decay property of the sequence $y_j$, the second RHS term of \eqref{eq:tightsum} is dominated by $A^{-1}$, independently of $n$. Finally, the third RHS term of \eqref{eq:tightsum} equals 
\[\mathds{1}_{\sum_j \left(\frac{\rho}{s}\right)^j \geq \frac{M}{A^2}}.\]
Hence, taking $M=A^3$ we get  
\[\sup_{n \in \N^*}   \PP_y \left( \sum_j \frac{y_j^{(n)}}{\Vert \alpha \ell_{n,j} \Vert} \geq M \right) \underset{M\to +\infty}{\longrightarrow} 0\]
and the proof is complete.
\end{proof}

\vspace*{1cm}

\begin{proof}[Proof of Theorem~\ref{thm:convergence} $(ii)$]
First, let us prove the convergence in law of $\xi_{n,\alpha}$ to $\widetilde{\xi}_\infty$ in the topology of pointwise convergence. Let $s \in \C$. For the sake of simplicity we write for all $n$ and $k$, $X_n = \xi_{n,\alpha} (s) = X_{n,k} Y_{n,k}$ where $X_{n,k}$ is the product of $\frac{\mathrm{e}^{2i\pi \left(\frac{s}{n} + \alpha \right)\ell_{n,j}} -1}{\mathrm{e}^{2i\pi \alpha \ell_{n,j}} -1}$ for $j\leq k$ and $Y_{n,k}$ is the remaining infinite product $j>k$. \\

By Proposition~\ref{prop:convpart2}, conditionally on $y$, for all $\varepsilon >0$, 
\[\PP_y (\vert Y_{n,k} -1 \vert \geq \varepsilon ) \leq \delta (\varepsilon , k)\]
where $\delta (\varepsilon , k)$ is independent of $n$ and tends to $0$ as $k$ goes to infinity. Moreover, the sequence $(X_{n,k})_{n,k\geq 1}$ is tight. Thus, 
\begin{align*}
\PP_y (\vert X_n - X_{n,k} \vert \geq \sqrt{\varepsilon} ) &\leq \PP_y (\vert X_n - X_{n,k} \vert \geq \varepsilon \vert X_{n,k} \vert ) + \PP_y ( \vert X_{n,k} \vert \geq \frac{1}{\sqrt{\varepsilon}}) \\
	&\leq \delta (\varepsilon ,k) + \eta (\varepsilon )
\end{align*}  
where $\eta (\varepsilon )$ is independent of $n$ and $k$ and tends to $0$ with $\varepsilon$.

Besides, by Proposition~\ref{prop:convpart1}, conditionally on $y$, $(X_{n,k})_{n\geq 1}$ converges in distribution to 
\[X_{\infty,k}:= \prod_{j\leq k} \frac{\mathrm{e}^{2i\pi s y_j} - u_j}{1-u_j}.\]

Let $\Phi$ be a bounded Lipschitz function on $\C$ with Lipschitz constant equal to $1$. Then, denoting $\E_y (\cdot ) = \E (\cdot  \ \vert \ y)$,
\begin{align*}
\vert \E_y (\Phi (X_n)) - \E_y (\Phi (X_{\infty , k})) \vert &\leq \E_y \vert \Phi (X_n) - \Phi (X_{n,k}) \vert + \vert \E_y (\Phi (X_{n,k})) - \E_y (\Phi (X_{\infty ,k})) \vert \\
	&\leq \sqrt{\varepsilon} + 2\max \vert \Phi \vert (\delta (\varepsilon , k) + \eta (\varepsilon )) + \nu (n,k) 
\end{align*}
where $\nu (n,k)$ tends to $0$ as $n$ goes to infinity. \\
Let $V$ be an accumulation point of the sequence $(\E_y (\Phi (X_n)))_{n\geq 1}$. Then
\[\vert V - \E_y (\Phi (X_{\infty , k})) \vert \leq \sqrt{\varepsilon} + 2\max \vert \Phi \vert (\delta (\varepsilon , k) + \eta (\varepsilon )).\]
Letting $k\to \infty$ gives
\[\vert V - \E_y (\Phi (X_\infty )) \vert \leq \sqrt{\varepsilon} + 2\max \vert \Phi \vert \eta (\varepsilon ),\]
where $X_\infty := \widetilde{\xi}_\infty (s)$, and finally letting $\varepsilon \to 0$ we get $V=\E_y (\Phi (X_\infty ))$. Consequently, $(\E_y (\Phi (X_n)))_{n\geq 1}$ has a unique accumulation point, so it converges to $\E_y (\Phi (X_\infty ))$. Conditionally on $y$, this proves the convergence in distribution of $X_n$ to $X_\infty$, that is to say $\xi_{n,\alpha}$ converges in distribution to $\widetilde{\xi}_\infty$ at every fixed point. For any finite number of points, we deduce the same result (instead of $X_n (s)$ we take $(X_n (s_1) , \dots , X_n (s_k))$ and apply a similar reasoning). It follows that, conditionally on $y$, $\xi_{n, \alpha}$ converges in distribution to $\widetilde{\xi}_\infty$ in the topology of pointwise convergence.  

Now, we want to extend to the topology of uniform convergence. By Proposition~\ref{prop:convpart3}, for every compact set $K$, the sequence $(\sup_K \vert \xi_{n,\alpha }\vert)_{n\geq 1}$ is tight. As the functions $\xi_{n,\alpha}$ are holomorphic, a straightforward consequence of Cauchy's differentiation formula is that $(\sup_K \vert \xi_{n,\alpha }^\prime \vert)_{n\geq 1}$ is also tight. In particular, for all $\varepsilon >0$ and for all $A$ large enough, 
\[\limsup_{n\to \infty} \PP (\sup_K \vert \xi_{n,\alpha}^\prime \vert \geq A ) < \varepsilon.\]
The mean value inequality implies that for all $\delta >0$, there exists $\Delta >0$ such that 
\[\limsup_{n\to \infty} \PP (\sup\limits_{\substack{\vert s_1 - s_2\vert < \Delta \\ s_1, s_2 \in K}} \vert \xi_{n,\alpha} (s_1) - \xi_{n,\alpha} (s_2) \vert \geq \delta ) < \varepsilon\]
(taking $\Delta := \delta/A$). This is exactly the property given in Definition 2.4 in \cite{chhaibi2017limiting} applied to the functions $\Phi_n = \xi_{n,\alpha}$ for $n \geq 1$: they are said to be \textbf{in probability compact-equicontinuous}. Finally, Lemma 2.5 of \cite{chhaibi2017limiting} allows us to conclude that conditionally on $y$, $\xi_{n,\alpha}$ converges also in law in the topology of uniform convergence on compact sets.   
In other words, for all bounded functional $F$ from $\mathcal{C}(\C$,$\C$) to $\C$, continuous with respect to the topology of the uniform convergence on compact sets,
\[\E \left[ F(\xi_{n,\alpha })  \ \vert \ y  \right] \overset{\text{a.s.}}{\underset{n\to \infty}{\longrightarrow}} \E \left[ F(\widetilde{\xi}_\infty ) \ \vert \ y  \right].\] 
Hence by dominated convergence 
\[\E \left[ F(\xi_{n,\alpha })  \right] \underset{n\to \infty}{\longrightarrow} \E \left[ F(\widetilde{\xi}_\infty ) \right].\]
\end{proof}

\section{Properties of the limiting function}
\label{sec:properties}

In the following proposition we show that $\widetilde{\xi}_\infty$ has order one, in the sense of entire functions. In other words, the infimum of all $\beta$ such that $\widetilde{\xi}_\infty (z)= \mathcal{O}(\exp (\vert z \vert^\beta ))$ as $\vert z \vert \to \infty$ equals one. The bound $\mathcal{O}(\exp (c\vert z \vert ))$ for this model can be compared to the one for the unitary case $\mathcal{O}(\exp (c\vert z \vert  \log \vert z \vert ))$ established by Chhaibi, Najnudel and Nikeghbali in \cite{chhaibi2017circular}.

\begin{prop}\label{prop:xi}
For all $\varepsilon >0$, there exists a random number $C_\varepsilon>0$ such that for all $z \in \C$, 
\begin{equation}\label{eq:majxiinfty}
\vert \widetilde{\xi}_\infty (z) \vert \leq C_\varepsilon \mathrm{e}^{(2\pi + \varepsilon) \vert z \vert }.
\end{equation}
\end{prop}

\begin{proof}
The bound is clear for $\vert z \vert < 1$ since $\widetilde{\xi}_\infty$ is locally bounded. \\
Assume $\vert z \vert\geq 1$. Let $\eta>0$. Set $k= \max \{j\geq 1 \ : \ 2\pi \vert z \vert y_j \geq \eta \} \vee 0$ and $k_1=  \max \{j\geq 1 \ : \ 2\pi  y_j \geq \eta \}\vee 0$. Note that $k\geq k_1$, with $k_1$ independent of $z$. We distinguish between two regimes depending on whether $j$ is lower or greater than $k$.
\begin{itemize}
\item \underline{$j > k$}:
	\begin{align*}
	\prod_{j=k+1}^{+\infty} \left\vert \frac{\mathrm{e}^{2 i\pi z y_j} - u_j}{1-u_j} \right\vert &\leq 				\prod_{j=k+1}^{+\infty} \left( 1 + \frac{2\pi \vert z \vert y_j}{\vert 1-u_j \vert} \exp (2\pi \vert z \vert y_j  ) \right) \\
		&\leq \prod_{j=k+1}^{+\infty} \left( 1 + \frac{2\pi \vert z \vert y_j}{\vert 1-u_j \vert} \exp (\eta  ) \right) \\
		&\leq \exp \left( 2\pi \vert z \vert \exp( \eta ) \sum_{j=k+1}^{+\infty}  \frac{y_j}{\vert 1-u_j \vert} \right) \\
		&\leq \exp \left( 2\pi \vert z \vert \exp( \eta ) \sum_{j=k_1+1}^{+\infty}  \frac{y_j}{\vert 1-u_j \vert} \right).
	\end{align*}
	Moreover we have seen from \eqref{eq:yjuj} that $\sum\limits_{j=1}^{+\infty}  \frac{y_j}{\vert 1-u_j \vert} < \infty$, and furthermore if $\eta$ tends to zero then $k_1$ goes to infinity, so we can chose $\eta$ sufficiently close to $0$ such that 
	$2\pi \exp (\eta ) \sum\limits_{j=k_1+1}^{+\infty}  \frac{y_j}{\vert 1-u_j \vert} \leq \frac{\varepsilon}{2}$ and then 
	\begin{equation}\label{eq:majxiinfty1}
	\prod_{j=k+1}^{+\infty} \left\vert \frac{\mathrm{e}^{2i\pi z y_j} - u_j}{1-u_j} \right\vert \leq \exp \left( \frac{\varepsilon}{2} \vert z \vert \right).
	\end{equation}
\item \underline{$j \leq k$} (case to be considered only when $k\neq 0$):
	As for all $m\in \N^*$,
	\[\left\vert \frac{\mathrm{e}^{2i\pi z y_m} - u_m}{1-u_m} \right\vert \leq \frac{\mathrm{e}^{2\pi \vert z \vert y_m} +1}{\vert 1 - u_m \vert } \leq \frac{2\mathrm{e}^{2\pi \vert z \vert y_m}}{\vert 1 - u_m \vert }, \]
	then
	\begin{align}\label{eq:majxiinfty2}
	\begin{split}
	\prod_{j=1}^k \left\vert \frac{\mathrm{e}^{2i \pi z y_j} - u_j}{1-u_j} \right\vert &\leq  \left( \frac{2}{\min\limits_{1\leq j \leq k}  \vert 1 -u_j \vert  }     \right)^k  \exp \left( 2\pi \vert z \vert \sum_{j=1}^k y_j   \right) \\
		&\leq \left( \frac{2}{\min\limits_{1\leq j \leq k}  \vert 1 -u_j \vert  }     \right)^k  \mathrm{e}^{2\pi \vert z \vert}. 
	\end{split}
	\end{align}
	It just remains to show that
	\begin{equation}\label{eq:majxiinfty3}
	\left( \frac{2}{\min\limits_{1\leq j \leq k}  \vert 1 -u_j \vert  } \right)^k \leq C_\varepsilon \exp \left( \frac{\varepsilon}{2} \vert z \vert \right)
	\end{equation}
	where $C_\varepsilon$ is a random number depending on $\varepsilon$ but not on $z$. Lemma~\ref{lem1} with $\alpha=3$ gives 
	\[\left( \frac{2}{\min\limits_{1\leq j \leq k}  \vert 1 -u_j \vert  } \right)^k \leq \left( \frac{2}{C_1} k^{3} \right)^k.\] 
	Moreover, using the assumption on the sequence $(y_j)$, there exists $r \in (0,1)$ and a (random) number $C>0$ such that for all $j$, $y_j \leq C r^j$, thus 
	\[k\leq \max \left\{ j\geq 1 : \ C r^j \geq \frac{\eta}{2\pi \vert z \vert} \right\} \leq \frac{\log \left(\frac{2\pi C \vert z \vert}{\eta} \right)}{- \log r} \ll \log (\vert z \vert +1 )\]
	Hence 
	\[ \log \left(\left( \frac{2}{C_1} k^{3} \right)^k \right) \ll \log (\vert z \vert +1) ( 1+ \log \log (\vert z \vert +1) ) \underset{\vert z \vert \to \infty}{=} o(\vert z \vert),\]
	which gives the existence of $C_\varepsilon$. \\
	Consequently, \eqref{eq:majxiinfty1}, \eqref{eq:majxiinfty2} and \eqref{eq:majxiinfty3} jointly give \eqref{eq:majxiinfty}.	
\end{itemize}	
\end{proof}

Up to $\varepsilon$, the bound we provide in Proposition~\ref{prop:xi} is sharp. Indeed, we have the following result:

\begin{prop}
For all $\varepsilon \in (0, 2\pi)$, there exists a random number $c_\varepsilon>0$ such that for all $x \geq 0$, 
\begin{equation}
\vert \widetilde{\xi}_\infty (-ix) \vert \geq c_\varepsilon \mathrm{e}^{(2\pi - \varepsilon) x }.
\end{equation}
\end{prop}
\begin{proof}
Let $x\geq 0$. For all $k\in \N^*$,
if $x \geq \frac{1}{2\pi y_k}$ then  
\[\left\vert \frac{\mathrm{e}^{2\pi x y_k} - u_k}{1-u_k}  \right\vert \geq 1 \geq \frac{1}{2} \vert \mathrm{e}^{2\pi x y_k} - u_k \vert \geq \frac{1}{2} ( \mathrm{e}^{2\pi x y_k} - 1)  \geq \frac{1}{4} \mathrm{e}^{2\pi x y_k}  ,\]
and if $0\leq x \leq \frac{1}{2\pi y_k}$ then 
\[ \left\vert \frac{\mathrm{e}^{2\pi x y_k} - u_k}{1-u_k}  \right\vert \geq  1 \geq \frac{1}{e} \mathrm{e}^{2\pi x y_k} \geq  \frac{1}{4} \mathrm{e}^{2\pi x y_k}. \]
Let $\varepsilon \in (0,2\pi)$. Set $p=\min \{ k\geq 1 \ : \  \sum\limits_{j=1}^{k} y_j \geq 1 - \frac{\varepsilon}{2\pi} \}$. Then
\begin{align*}
\vert \widetilde{\xi}_\infty (-ix) \vert &= \prod_{j=1}^{\infty} \left\vert \frac{\mathrm{e}^{2\pi x y_j} - u_j}{1-u_j}    \right\vert \\
	&\geq \prod_{j=1}^p \left\vert \frac{\mathrm{e}^{2\pi x y_j} - u_j}{1-u_j} \right\vert \\
	&\geq \left( \frac{1}{4} \right)^p \exp \left( 2\pi x \sum_{j=1}^p y_j \right) \\
	&\geq \left( \frac{1}{4} \right)^p \mathrm{e}^{(2\pi - \varepsilon) x }.
\end{align*}
\end{proof}

\section{More general central measures}
\label{sec:moregeneral}

In this section we treat the case of central measures such that $\PP \left( \sum_{j\geq 1} y_j <1 \right) > 0$.

Let $\lambda = (\lambda_j)_{j\geq 1}$ be a sequence of decreasing non-negative real numbers summing to $\lambda_0 \in (0,1)$, and let $E_\lambda = \bigsqcup\limits_{j=1}^\infty \mathcal{C}_j \sqcup S$ be the disjoint union of circles $\mathcal{C}_j$, where for all $j$, $\mathcal{C}_j$ has perimeter $\lambda_j$, and of a segment $S$ of length $1-\lambda_0$. Let $x=(x_k)_{k\geq 1} \in (\E_\lambda )^{\infty}$. As in Section~\ref{sec:gen}, for all $n$, one can define a permutation $\sigma_n (\lambda, x)$ as follows: for all $k \in \{ 1, \dots , n\}$, 
\begin{itemize}
	\item if $x_k$ is on a circle $\mathcal{C}_j$, then the image of $k$ by $\sigma_n (\lambda , x)$ is the index of the first point in $\mathcal{C}_j$ and $\{x_1, \dots , x_n\}$ we encounter after $x_k$ following $\mathcal{C}_j$ counterclockwise,
	\item if $x_k$ is in $S$, then $k$ is a fixed point of $\sigma_n (\lambda , x)$.
\end{itemize}
To illustrate, if the $\lambda_j$ equal $3^{-j}$ and if the six first $x_k$ are distributed on $E_\lambda$ as shown
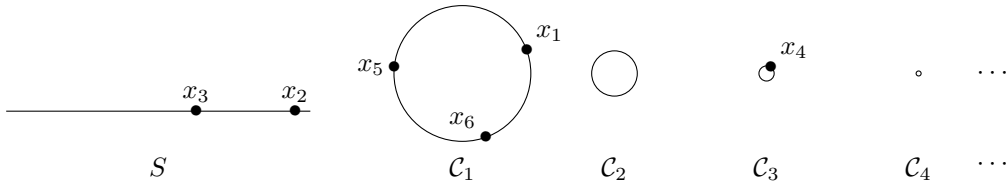
\begin{center}
\begin{tikzpicture}
\draw (0,0) circle (0.9) ;
\draw (2,0) circle (0.3) ;
\draw (4,0) circle (0.1) ;
\draw (6,0) circle (0.033) ;
\draw (7,0) node {$\cdots$} ;
\draw (-6,-0.5)--(-2,-0.5) ;
\draw (0,-1) node[below] {$\mathcal{C}_1$} ;
\draw (2,-1) node[below] {$\mathcal{C}_2$} ;
\draw (4,-1) node[below] {$\mathcal{C}_3$} ;
\draw (6,-1) node[below] {$\mathcal{C}_4$} ;
\draw (7,-1) node[below] {$\cdots$} ;
\draw (-4,-1) node[below] {$S$} ;
\draw (20:0.9) node {$\bullet$} node[above right] {$x_1$};
\draw (175:0.9) node {$\bullet$} node[left] {$x_5$};
\draw (-70:0.9) node {$\bullet$} node[above left] {$x_6$};
\draw (55:0.1)+(4,0) node {$\bullet$} node[above right] {$x_4$};
\draw (-3.5,-0.5) node {$\bullet$} node[above] {$x_3$};
\draw (-2.2,-0.5) node {$\bullet$} node[above] {$x_2$};
\end{tikzpicture}
\end{center}
then, 
\begin{align*}
\sigma_1 (\lambda , x) & = (1) \\
\sigma_2 (\lambda , x) & = (1)(2) \\
\sigma_3 (\lambda , x) & = (1)(2)(3) \\
\sigma_4 (\lambda , x) & = (1)(2)(3)(4) \\
\sigma_5 (\lambda , x) & = (1 \ 5)(2)(3)(4) \\
\sigma_6 (\lambda , x) & = (1 \ 5 \ 6)(2)(3)(4).
\end{align*}
Naturally, in the same way as for Proposition~\ref{prop:couplage}, we have the following proposition.
\begin{prop}
The sequence $\sigma_\infty (\lambda, x)=(\sigma_n (\lambda, x))_{n\geq 1}$ is a virtual permutation. Moreover, if $\lambda$ follows any arbitrary distribution $p$ on $\nabla$, and if conditionally on $\lambda$ the points $x_k$ are i.i.d following the uniform distribution on $E_\lambda $, then $\sigma_\infty (\lambda, x)$ follows the central measure on $\mathfrak{S}$ corresponding to $p$.
\end{prop}

Let $p$ be any probability measure on $\nabla$, and let $(y_j)_{j\geq 1}$ be a random vector following the distribution $p$. Introduce the random variable $y_0:=\sum\limits_{j=1}^{+\infty} y_j$.

For all $n$ and $j$, denote $\ell_{n,j}:= \# \{ k \in \{1 , \dots n \} \ : \ x_k \in \mathcal{C}_j \}$, and $p_n:= \# \{ k \in \{1 , \dots n \} \ : \ x_k \in S \}$.

Let $\alpha$ be an irrational number between $0$ and $1$. The expressions of $\widetilde{\xi}_n$ and $\xi_{n,\alpha}$ (see \eqref{eq:xi1} and \eqref{eq:xialpha1}) become 
\begin{equation}
\widetilde{\xi}_n (z) = \left(\prod_{\substack{j\geq 1 \\ \ell_{n,j}>0}} \frac{\mathrm{e}^{2i\pi z y_j^{(n)}} - u_j}{1-u_j}\right) \left(\prod_{k=1}^{p_n} \frac{\mathrm{e}^{2i\pi \frac{z}{n}} - v_k}{1-v_k} \right)
\end{equation}
where the $u_j$ and the $v_k$ are independent random variables uniformly distributed on the unit circle,
and
\begin{equation}
\xi_{n,\alpha } = \left( \prod_{\substack{j\geq 1 \\ \ell_{n,j}>0}} \frac{\mathrm{e}^{2i\pi \left(\frac{z}{n} + \alpha \right)\ell_{n,j}} -1}{\mathrm{e}^{2i\pi \alpha \ell_{n,j}} -1} \right) \left(\frac{\mathrm{e}^{2i\pi \left( \frac{z}{n} + \alpha \right)} - 1}{\mathrm{e}^{2i\pi \alpha}- 1} \right)^{p_n} .
\end{equation}

\begin{theoreme}\label{thm:convergence2}
Assume that $\sigma$ is generated by the coupling described above for a distribution with exponential decay $p$ on $\nabla$. Then we have the following convergences in distribution:
\begin{enumerate}[label=(\roman*)]
	\item 
\[\widetilde{\xi}_n (z) \underset{n\to\infty}{\Longrightarrow} \left(\prod_{j=1}^{+\infty} \frac{\mathrm{e}^{2i\pi z y_j} - u_j}{1-u_j}\right) \mathrm{e}^{i\pi z (1-y_0)} \prod_{k\in \Z} \left( 1 - \frac{z}{w_k} \right)\]
where $\{w_k : k\in \Z\}$ are points of a Poisson process with intensity $1-y_0$ on $\R$ (if $y_0=1$, we make the convention $ \prod_{k\in \Z} \left( 1 - \frac{z}{w_k} \right) =1$). 
	\item For all irrational number $\alpha$ of finite type,
\[\xi_{n,\alpha} (z) \underset{n\to\infty}{\Longrightarrow} \left(\prod_{j=1}^{+\infty} \frac{\mathrm{e}^{2i\pi z y_j} - u_j}{1-u_j}\right) \mathrm{e}^{i\pi z (1-y_0)\left(1- \frac{i}{\tan (\pi \alpha)} \right)}.\]
\end{enumerate}
\end{theoreme}

\begin{remarque}
In Theorem~\ref{thm:convergence2}, the product $\prod\limits_{k\in \Z} \left( 1 - \frac{z}{w_k} \right)$ is not absolutely convergent. It has to be understood as 
\[ \left( 1 - \frac{z}{w_0} \right)\prod\limits_{k\geq 1} \left( 1 - \frac{z}{w_k} \right) \left( 1 - \frac{z}{w_{-k}} \right),\]
where the points of the Poisson process $\{w_k : k\in \Z\}$ are labelled as follows:
\[\dots < w_{-2} < w_{-1} < 0 \leq w_0 < w_1 < w_2 < \dots .\] 
The fact that this product with this ordering is convergent is a direct consequence of Lemma~\ref{lem:approxPoisson} with $\varepsilon < 1/2$.
\end{remarque}

\subsection{Proof of Theorem~\ref{thm:convergence2} (i)}

For all $k\geq 1$, write $v_k = \mathrm{e}^{2i\pi \Phi_k}$ where the $\Phi_k$ are independent and uniformly distributed on $[0,1)$. Then
\begin{align*}
\prod_{k=1}^{p_n} \frac{\mathrm{e}^{2i\pi \frac{z}{n}} - v_k}{1-v_k} &= \prod_{k=1}^{p_n} \frac{\mathrm{e}^{2i\pi \frac{z}{n}} - \mathrm{e}^{2i\pi \Phi_k}}{1-\mathrm{e}^{2i\pi \Phi_k}} \\
	&= \mathrm{e}^{i\pi z\frac{p_n}{n}} \prod_{k=1}^{p_n} \frac{\sin \left( \pi \left( \Phi_k - \frac{z}{n} \right)\right)}{\sin (\pi \Phi_k )} \\
	&= \mathrm{e}^{i\pi z\frac{p_n}{n}} \prod_{k=1}^{p_n} \frac{\Phi_k - \frac{z}{n}}{\Phi_k} \lim_{B\to +\infty} \prod_{0<\vert j \vert \leq B } \frac{1- \frac{\Phi_k - \frac{z}{n}}{j}}{1- \frac{\Phi_k}{j}} \\
	&=  \mathrm{e}^{i\pi z\frac{p_n}{n}} \prod_{k=1}^{p_n} \prod_{j\in \Z} \left( 1 - \frac{z}{n(\Phi_k -j)} \right)
\end{align*}
using the product expansion of the sine function for the penultimate equality, that is
\[\sin (\pi z) = \pi z \prod_{j=1}^{+\infty} \left( 1- \frac{z^2}{j^2} \right) = \pi z \lim_{B\to +\infty} \prod_{0<\vert j \vert \leq B } \left( 1- \frac{z}{j} \right). \]
\begin{remarque}
The product $\prod\limits_{j\in \Z} \left( 1 - \frac{z}{n(\Phi_k -j)} \right)$ is not absolutely convergent. We write it like this for convenience of notation, and has to be understood as $\lim\limits_{B\to +\infty} \prod\limits_{0\leq \vert j \vert \leq B } \left( 1 - \frac{z}{n(\Phi_k -j)} \right)$.
\end{remarque}

Denote by $w_{n,\ell}$ the points $n(\Phi_k -j)$ for $k \in [\![1, p_n ]\!]$ and $j\in \Z$ (defined for example as follows: 
for all $\ell \in \Z$, $w_{n,\ell}:=n(\Phi_k -j)$ where $(k,j)$ is the unique couple of $[\![1, p_n ]\!] \times \Z$ such that $\ell=k+jp_n$).
The order of labelling for the points $w_{n,\ell}$ does not matter for what we use in the sequel. 

Let $\mu_n := \sum\limits_{k\in \Z} \delta_{w_{n,k}}$ be the empirical measure associated with the point process of the $w_{n,k}$. 

\begin{lem}\label{lem:topoisson}
The empirical measure $\mu_n$ converges vaguely to $\mu_\infty$, where $\mu_\infty$ is the empirical measure associated with the points of a Poisson process with intensity $1-y_0$ on $\R$. The convergence holds for all compactly supported test functions from $\R$ to $\C$ (measurable but not necessarily continuous).
\end{lem}
\begin{proof}
Let $f : \R \to \C$ such that $\mathrm{supp} f \subset [-M, M]$, $M>0$. Let $t\in \R$. Let $n>2M$.
First note that by periodicity of the points $w_{n,k}$, 
\[\E \left( \mathrm{e}^{it\sum\limits_k f (w_{n,k})} \right) = \E \left( \mathrm{e}^{it\sum\limits_{k=1}^{X_n} f (\phi_k)} \right) \]
where $X_n$ is a random variable which counts the number of the points $w_{n,k}$ lying in $[-M,M]$, and where the $\phi_k$ are i.i.d random variables uniformly chosen on $[-M,M]$, independently of $X_n$. Moreover, $X_n$ is binomial of parameters $p_n$ and $2M/n$. Hence,
\begin{align*}
\E \left( \mathrm{e}^{it\sum\limits_k f (w_{n,k})} \right) &= \E \left( \E (\mathrm{e}^{it f(\phi_1 )})^{X_n} \right) \\
	&= \sum_{k=0}^{p_n} \E (\mathrm{e}^{it f(\phi_1 )})^k \binom{p_n}{k} \left(\frac{2M}{n}\right)^k \left( 1 - \frac{2M}{n} \right)^{p_n -k}  \\
	&= \left( 1+ \frac{2M}{n} (\E (\mathrm{e}^{it f(\phi_1 )})-1)\right)^{p_n} \\
	&= \left( 1 + \frac{1}{n} \int_{-M}^M \left( \mathrm{e}^{it f(x)} -1 \right) \mathrm{d}x \right)^{p_n} \\
	&\underset{n\to +\infty}{=} \exp \left( (1-y_0) \int_{-M}^M \left( \mathrm{e}^{it f(x)} -1 \right) \mathrm{d}x   \right) + o(1),
\end{align*}
since $p_n/n \to 1-y_0$ almost surely. Thus, the Fourier transform of $\sum\limits_k f (w_{n,k})$ converges to the Fourier transform of $T:=\sum\limits_{x\in N} f(x)$, where $N$ is a homogeneous Poisson point process of parameter $1-y_0$ (for example the expression of the Fourier transform of $T$ can be provided using the Campbell theorem), which gives the claim.
\end{proof}

\begin{prop}\label{prop:topoisson1}
For all $A$, 
\[ \prod_{\vert w_{n,k} \vert < A} \left( 1 - \frac{z}{w_{n,k}} \right) \underset{n\to +\infty}{\Longrightarrow } \prod_{\vert w_k \vert < A} \left( 1 - \frac{z}{w_k} \right)\]
where $\{w_k : k\in \Z\}$ are points of a Poisson process with intensity $1-y_0$ on $\R$.
\end{prop}
\begin{proof}
Let $A>0$. Let $\mathcal{M}$ denote the space of locally finite measures of the form $\sum\limits_k \delta_{\alpha_k}$ for some arbitrary real numbers $\alpha_k$. Let $F$ be the functional defined from $\mathcal{M}$ to $\mathcal{C}(\C , \C)$ by
\[F \left(\sum\limits_k \delta_{\alpha_k} \right) = \prod\limits_k \left( 1- \frac{z}{\alpha_k} \right)\mathds{1}_{0<\vert \alpha_k \vert < A}. \]
$F$ is continuous at every measure which does not charge $-A$, $0$, and $A$. Since almost surely $\mu_\infty$ (the empirical measure associated with the Poisson process $\{w_k : \ k\in \Z \}$) does not charge these three points, $F$ is continuous in $\mu_\infty$. By Lemma~\ref{lem:topoisson} and the continuous mapping theorem we deduce $F(\mu_n ) \to F(\mu_\infty )$, which gives the claim.
\end{proof}

\begin{prop}\label{prop:topoisson2}
For all $\varepsilon >0$, for all compact subsets $K$ of $\C$,
\[\sup_{n\in \N^*} \PP \left( \sup_{z\in K} \left\vert \prod_{ \vert w_{n,k} \vert \geq A } \left( 1- \frac{z}{w_{n,k}}  \right) -1 \right\vert \geq \varepsilon \right) \underset{A \to +\infty}{\longrightarrow} 0.\] 
\end{prop}
\begin{proof}
Let $K$ a compact subset of $\C$ of diameter $D$ for the uniform norm, and let $A>2D$. For all $z\in K$, 
\begin{equation}\label{eq:majProdtoPoisson}
\prod_{ \vert w_{n,k} \vert \geq A } \left( 1- \frac{z}{w_{n,k}}  \right) = \exp \left( -z \sum_{ \vert w_{n,k} \vert \geq A} \frac{1}{w_{n,k}} \right) \exp \left( \mathcal{O}_K \left( \sum_{ \vert w_{n,k} \vert \geq A} \frac{1}{w_{n,k}^2}  \right) \right).
\end{equation}
The sum $\sum\limits_{ \vert w_{n,k} \vert \geq A} \frac{1}{w_{n,k}}$ is not absolutely convergent. Let $B>A$. Integrating by parts, 
\begin{align*}
\sum_{A\leq \vert w_{n,k} \vert \leq B} \frac{1}{w_{n,k}} &= \int_{-B}^{-A} \frac{1}{x} \mathrm{d}\mu_n (x) + \int_{A}^{B} \frac{1}{x} \mathrm{d}\mu_n (x)  \\
	&= \frac{\mu_n [A, B] -\mu_n [-B, -A]}{B} + \int_{A}^{B} \frac{\mu_n [A, x]-\mu_n [-x, -A]}{x^2} \mathrm{d}x.
\end{align*}
As for all $a,b\in \R$, $\left\vert \mu_n [a,b] - \frac{(b-a)}{n}p_n \right\vert \leq 2p_n$, then $\lim\limits_{B\to +\infty} \frac{\mu_n [A, B] -\mu_n [-B, -A]}{B} = 0$, and we get
\[\sum_{ \vert w_{n,k} \vert \geq A} \frac{1}{w_{n,k}}  = \int_{A}^{+\infty} \frac{\mu_n [A, x]-\mu_n [-x, -A]}{x^2} \mathrm{d}x. \]
Hence 
\begin{align*}
\E \left\vert \sum_{ \vert w_{n,k} \vert \geq A} \frac{1}{w_{n,k}} \right\vert  \leq \int_{A}^{+\infty} \frac{1}{x^2} \E \vert \mu_n [A, x]-\mu_n [-x, -A] \vert \mathrm{d}x,
\end{align*}
with for all $x>A$, by Cauchy-Schwarz inequality, 
\begin{align*}
(\E \vert \mu_n [A, x]-\mu_n [-x, -A] \vert)^2 &\leq \E ((\mu_n [A, x]-\mu_n [-x, -A])^2) \\
	&= \V (\mu_n [A, x]-\mu_n [-x, -A]) \\
	&\leq 2 (\V (\mu_n [A, x]) + \V (\mu_n [-x, -A])) \\
	&= 4 \V (\mu_n [A,x])
\end{align*}
since $\mu_n [A, x]$ and $\mu_n [-x, -A]$ are equally distributed (consequence of the fact that the $\Phi_k$ are uniformly distributed on $[0,1)$). \\
As each interval of the form $[jn, (j+1)n)$ contains exactly $p_n$ points, and these points are uniformly distributed, then $\V (\mu_n [A,x])$ is the variance of a binomial random variable of parameters $p_n$ and $\left\{\frac{x-A}{n} \right\}$, that is   
\begin{align*}
 \V (\mu_n [A,x]) &= p_n \left\{\frac{x-A}{n} \right\} \left( 1-  \left\{\frac{x-A}{n} \right\} \right) \\
 	&\leq p_n \frac{x-A}{n} \\
 	&\leq x-A.
\end{align*}
We deduce
\begin{align*}
\E \left\vert \sum_{ \vert w_{n,k} \vert \geq A} \frac{1}{w_{n,k}} \right\vert  &\leq 2\int_{A}^{+\infty} \frac{\sqrt{x-A}}{x^2}  \mathrm{d}x  = \frac{\pi}{\sqrt{A}} \underset{A \to +\infty} {\longrightarrow} 0.
\end{align*}
Moreover,
\[\E \sum_{ \vert w_{n,k} \vert \geq A} \frac{1}{w_{n,k}^2} = \int_{-\infty}^{-A} \frac{1}{x^2} \frac{p_n}{n} \mathrm{d}x +  \int_A^{+\infty} \frac{1}{x^2} \frac{p_n}{n} \mathrm{d}x \leq \frac{2}{A} \underset{A \to +\infty} {\longrightarrow} 0.\]
From \eqref{eq:majProdtoPoisson} we deduce that $z\mapsto \prod\limits_{ \vert w_{n,k} \vert \geq A } \left( 1- \frac{z}{w_{n,k}}  \right)$ converges in probability to $z\mapsto 1$ on every compact sets as $A$ goes to $+\infty$, uniformly in $n$, which gives the claim.
\end{proof}

Proposition~\ref{prop:topoisson1} and Proposition~\ref{prop:topoisson2} together show
\[ \prod_{k=1}^{p_n} \prod_{j\in \Z} \left( 1 - \frac{z}{n(\Phi_k -j)} \right) \underset{n\to +\infty}{\Longrightarrow} \prod_{k\in \Z} \left( 1 - \frac{z}{w_k} \right)\]
(see remark of Theorem~\ref{thm:convergence2} for the sense given to this last non-absolutely convergent product),
hence
\[\prod_{k=1}^{p_n} \frac{\mathrm{e}^{2i\pi \frac{z}{n}} - v_k}{1-v_k} \underset{n\to +\infty}{\Longrightarrow} \mathrm{e}^{i\pi z (1-y_0)}\prod_{k\in \Z} \left( 1 - \frac{z}{w_k} \right). \]

Finally, in the same manner as in the proof of point $(i)$ of Theorem~\ref{thm:convergence}, we prove the almost sure convergence of $\prod\limits_{\substack{j\geq 1 \\ \ell_{n,j}>0}} \frac{\mathrm{e}^{2i\pi z y_j^{(n)}} - u_j}{1-u_j}$ to $\prod\limits_{j=1}^{+\infty} \frac{\mathrm{e}^{2i\pi z y_j} - u_j}{1-u_j}$. Conditionally on $y$ and $(u_j)_{j\geq 1}$, Slutsky's theorem applies on the fonctional space $\mathcal{C}(\C, \C)$, which allows to conclude using the dominated convergence theorem. 

\subsection{Proof of Theorem~\ref{thm:convergence2} (ii)}

The proof of point $(ii)$ is much simpler. Indeed, it suffices to see that for all $n$ and for all $z$ in any compact subset $K$ of $\C$,
\begin{align*}
\left(\frac{\mathrm{e}^{2i\pi \left( \frac{z}{n} + \alpha \right)} - 1}{\mathrm{e}^{2i\pi \alpha}- 1} \right)^{p_n} &= \left( 1 + \frac{\mathrm{e}^{2i\pi \alpha}}{\mathrm{e}^{2i\pi \alpha}- 1} \left( \mathrm{e}^{2i\pi \frac{z}{n}} -1   \right) \right)^{p_n} \\
	&=\exp \left(  p_n \left[  \frac{2i\pi\mathrm{e}^{2i\pi \alpha}}{\mathrm{e}^{2i\pi \alpha}- 1}  \frac{z}{n} + \mathcal{O}\left( \frac{1}{n^2} \right)\right] \right) \\
	&=\exp \left( \frac{p_n}{n} \frac{2i\pi z\mathrm{e}^{2i\pi \alpha}}{\mathrm{e}^{2i\pi \alpha}- 1}  + \mathcal{O}\left( \frac{1}{n} \right) \right) \\
	&= \exp \left( (1-y_0) \frac{2i\pi z\mathrm{e}^{2i\pi \alpha}}{\mathrm{e}^{2i\pi \alpha}- 1} \right) +o(1) 
\end{align*}
uniformly in $z\in K$, and for all $n$ large enough depending on $\alpha$ and $K$. \\
Finally, simplifying $\frac{2i \mathrm{e}^{2i\pi \alpha}}{\mathrm{e}^{2i\pi \alpha}- 1} = \frac{\mathrm{e}^{i\pi \alpha}}{\sin (\pi \alpha )} = \frac{1}{\tan (\pi \alpha )} + i$ gives the claim.

\subsection{Properties of the limiting functions}

\begin{lem}\label{lem:approxPoisson}
Let $\varepsilon>0$. Almost surely, for all $k\in \Z$, 
\begin{equation}
w_k = \frac{k}{1-y_0} + \mathcal{O} \left( k^{\frac{1}{2}+\varepsilon} \right).
\end{equation}
\end{lem}
We omit the proof of this lemma since it is a classical result on Poisson processes (for instance, for $k\geq 1$, $w_k$ reads as a sum of $k$ i.i.d. exponential variables, so one can apply general results of moderate deviations, see \emph{e.g.} \cite{eichelsbacher2003moderate} with $b_k = k^{\frac{1}{2} + \varepsilon}$).

\begin{prop}
For all $\varepsilon >0$, there exists a random number $C>0$ such that for all $z \in \C$, 
\begin{equation}
\vert \widetilde{\xi}_\infty (z) \vert \leq \mathrm{e}^{ C \vert z \vert \log (2+\vert z \vert ) }.
\end{equation}
\end{prop}
\begin{proof}
First, following the same reasoning as in Section~\ref{sec:properties}, it is easy to check that for all $\varepsilon >0$, there exists a random number $C_\varepsilon >0$ such that for all $z\in \C$,
\[\left\vert \left(\prod_{j=1}^{+\infty} \frac{\mathrm{e}^{2i\pi z y_j} - u_j}{1-u_j}\right) \right\vert \leq C_\varepsilon \mathrm{e}^{(2\pi y_0 + \varepsilon) \vert z \vert }. \]
Hence there exists a random number $c>0$ such that for all $z \in \C$, 
 \[\left\vert \left(\prod_{j=1}^{+\infty} \frac{\mathrm{e}^{2i\pi z y_j} - u_j}{1-u_j}\right) \mathrm{e}^{i\pi z (1-y_0)} \right\vert \leq \mathrm{e}^{ c \vert z \vert }. \]
Thus, it is enough to show that there exists a random number $C>0$ such that for all $z\in \C$,
\begin{equation}\label{eq:majPoisson}
\left\vert \prod_{k\in \Z} \left( 1 - \frac{z}{w_k} \right) \right\vert \leq \mathrm{e}^{C \vert z \vert \log (2 + \vert z \vert ) }.
\end{equation}
To this end, we distinguish between two regimes of $k\neq 0$ in this product: $\vert k \vert \geq \vert z \vert$, and $1\leq \vert k \vert \leq \vert z \vert$. For the first regime, using the previous lemma with $\varepsilon = \frac{1}{3}$ gives
\[\left( 1 - \frac{z}{w_k} \right) \left( 1 - \frac{z}{w_{-k}} \right) = 1 + \mathcal{O} \left(  \frac{\vert z \vert k^{\frac{1}{2} + \frac{1}{3}} + \vert z \vert^2}{k^2} \right),  \]
hence
\begin{align}\label{eq:majPoisson1}
\begin{split}
\left\vert \prod_{k \geq \vert z \vert} \left( 1 - \frac{z}{w_k} \right) \left( 1 - \frac{z}{w_{-k}} \right) \right\vert &\leq \exp \left( \mathcal{O} \left(\sum_{k\geq \vert z \vert} \vert z \vert k^{-\frac{7}{6}} +  \vert z \vert^2 k^{-2} \right) \right) \\
	&= \exp (\mathcal{O}(\vert z \vert )).
\end{split}
\end{align} 
For the second regime, as $\left\vert \frac{w_k}{k} \right\vert$ is almost surely bounded from below (since $\frac{w_k}{k} = \frac{1}{1-y_0} + \mathcal{O} \left( k^{-\frac{1}{2}+\varepsilon}  \right)$), we have 
\[1- \frac{z}{w_k} = 1 + \mathcal{O}\left( \left\vert \frac{z}{k} \right\vert \right) \]
and it follows
\begin{align}\label{eq:majPoisson2}
\begin{split}
\left\vert \prod_{1\leq  k < \vert z \vert} \left( 1 - \frac{z}{w_k} \right) \left( 1 - \frac{z}{w_{-k}} \right) \right\vert &\leq \exp \left( \mathcal{O} \left(\sum_{1\leq  k < \vert z \vert}  \frac{\vert z \vert}{k} 
 \right) \right) \\
	&= \exp (\mathcal{O}(\vert z \vert \log (2+ \vert z \vert ))).
\end{split}
\end{align}
Furthermore, 
\begin{equation}\label{eq:majPoisson3}
\left\vert 1 - \frac{z}{w_0} \right\vert \leq \exp \left( \frac{\vert z \vert}{\vert w_0 \vert} \right) = \exp (\mathcal{O}(\vert z \vert ))
\end{equation}
since $w_0\neq 0$ almost surely. Combining \eqref{eq:majPoisson1}, \eqref{eq:majPoisson2} and \eqref{eq:majPoisson3}, we deduce the existence of a random number $C>0$ such that for all $z\in \C$ we have \eqref{eq:majPoisson}, and the proof is complete.
\end{proof}

\begin{prop}
For all $\varepsilon >0$, there exists a random number $C_\varepsilon>0$ such that for all $z \in \C$, 
\begin{equation}
\vert \xi_{\infty , \alpha} (z) \vert \leq C_\varepsilon \mathrm{e}^{ (\varepsilon + 2\pi (y_0 + (1-y_0)t_\alpha) \vert z \vert },
\end{equation}
where $t_\alpha = \frac{1}{2\sin(\pi \alpha )} \in \left( \frac{1}{2} , + \infty \right)$.
\end{prop}

The proof of the last proposition is very similar to the one for the case $\sum_{j=1}^{+\infty} y_j=1$ almost surely. We omit it here and refer to Section~\ref{sec:properties}.

\vspace*{1cm}

\textbf{Acknowledgements}: The author is grateful to his PhD advisor Joseph Najnudel for suggesting the problem and for a significant help on the topic.

\bibliographystyle{plain}
\bibliography{biblio}

\end{document}